\documentclass[psamsfonts]{amsart}

\usepackage{amssymb,amsfonts,amscd}

\usepackage[colorlinks,linktocpage]{hyperref}
\hypersetup{linkcolor=[rgb]{0,0,0.715}}
\hypersetup{citecolor=[rgb]{0,0.715,0}}

\newtheorem{thm}{Theorem}[section]

\newtheorem{prop}[thm]{Proposition}
\newtheorem{lem}[thm]{Lemma}
\newtheorem{conj}[thm]{Conjecture}
\newtheorem{quest}[thm]{Question}

\newtheorem*{thm1}{Theorem A}

\theoremstyle{definition}
\newtheorem{defn}[thm]{Definition}

\theoremstyle{remark}
\newtheorem{rem}[thm]{Remark}

\makeatletter
\let\c@equation\c@thm
\makeatother
\numberwithin{equation}{section}

\title[]{Nonnegative Ricci curvature, metric cones,\\ and virtual abelianness}
\author[]{Jiayin Pan}
\address[]{Fields Institute for Research in Mathematical Sciences, Toronto, Ontario, Canada.}
\email{jypan10@gmail.com}
\begin{document}
	
	\begin{abstract}
		Let $M$ be an open $n$-manifold with nonnegative Ricci curvature. We prove that if its escape rate is not $1/2$ and its Riemannian universal cover is conic at infinity, that is, every asymptotic cone $(Y,y)$ of the universal cover is a metric cone with vertex $y$, then $\pi_1(M)$ contains an abelian subgroup of finite index. If in addition the universal cover has Euclidean volume growth of constant at least $L$, we can further bound the index by a constant $C(n,L)$. 
	\end{abstract}
	
	\maketitle

We study the virtual abelianness/nilpotency of fundamental groups of open manifolds with $\mathrm{Ric}\ge 0$. According to the work of Kapovitch-Wilking \cite{KW}, these fundamental groups always have nilpotent subgroups with index at most $C(n)$ (also see \cite{Mil,Gro_poly}). In general, these fundamental groups may not contain any abelian subgroups with finite index, because Wei has constructed examples with torsion-free nilpotent fundamental groups \cite{Wei}. This is different from manifolds with $\mathrm{sec}\ge 0$, whose fundamental groups are always virtually abelian \cite{CG_soul}.

A question raised from here is, for an open manifold $M$ with $\mathrm{Ric}\ge0$, on what conditions is $\pi_1(M)$ virtually abelian? To answer this question, one naturally looks for indications from the geometry of nonnegative sectional curvature. Ideally, if the manifold $M$ fulfills some geometric conditions modeled on nonnegative sectional curvature, even in a much weaker form, then $\pi_1(M)$ may turn out to be virtually abelian. In other words, when $\pi_1(M)$ is not virtually abelian, some aspects of $M$ should be drastically different from the geometry of nonnegative sectional curvature.

We have explored this direction in \cite{Pan_es0,Pan_esgap}, from the viewpoint of escape rate. Recall that each element $\gamma$ in $\pi_1(M,p)$ can be represented by a geodesic loop at $p$, denoted by $c_\gamma$, with the minimal length in its homotopy class. If $M$ has $\mathrm{sec}\ge 0$, then all these representing loops must stay in a bounded ball; however, this property in general does not hold for nonnegative Ricci curvature. The escape rate measures how fast these loops escape from bounded balls:
$$E(M,p):=\limsup_{|\gamma|\to\infty}\dfrac{d_H(p,c_\gamma)}{|\gamma|},$$
where $|\gamma|$ is the length of $c_\gamma$ and $d_H$ is the Hausdorff distance. As the main result of \cite{Pan_esgap}, if $M$ satisfies $\mathrm{Ric}\ge 0$ and $E(M,p)\le \epsilon(n)$, then $\pi_1(M)$ is virtually abelian. We also mention that from the definition, the escape rate always takes values between $0$ and $1/2$. To the author's best knowledge, all known examples of open manifolds with $\mathrm{Ric}\ge 0$ have escape rate strictly less than $1/2$.

In this paper, we study how virtual abelianness/nilpotency is related to conic asymptotic geometry. Recall that an asymptotic cone of $M$ is the pointed Gromov-Hausdorff limit of a sequence 
$$(r_i^{-1}M,p)\overset{GH}\longrightarrow (Y,y),$$
where $r_i\to\infty$. We say that $M$ is \textit{conic at infinity}, if any asymptotic cone $(Y,y)$ of $M$ is a metric cone with vertex $y$. We do not assume the asymptotic cone to be unique in this definition. If the manifold has $\mathrm{sec}\ge 0$, then its asymptotic cone $(Y,y)$ is unique as a metric cone with vertex $y$. If the manifold has $\mathrm{Ric}\ge 0$ and Euclidean volume growth, then it is conic at infinity \cite{CC96}.
	
We state the main result of this paper.	
	
\begin{thm1}
	Let $(M,p)$ be an open $n$-manifold with $\mathrm{Ric}\ge 0$ and $E(M,p)\not=1/2$.\\
	(1) If its Riemannian universal cover is conic at infinity, then $\pi_1(M)$ is virtually abelian.\\
	(2) If its Riemannian universal cover has Euclidean volume growth of constant at least $L$, then $\pi_1(M)$ has an abelian subgroup of index at most $C(n,L)$, a constant only depending on $n$ and $L$.
\end{thm1}	

The contra-positive of Theorem A(1) shows that the nilpotency of $\pi_1(M)$ leads to asymptotic geometry of the universal cover that is very different from the one with nonnegative sectional curvature (also see Conjecture \ref{conj_nil_dim}).

Before proceeding further, we make some comments about the conditions in Theorem A.

We emphasize that in Theorem A(1), the conic at infinity condition is imposed on the Riemannian universal cover of $M$, not $M$ itself. In fact, Wei's example \cite{Wei} has the half-line $([0,\infty),0)$ as the unique asymptotic cone of $M$. Therefore, in general, virtual abelianness does not hold when $M$ is conic at infinity. 

Regarding the condition $E(M,p)\not=1/2$ in Theorem A, it is unclear to the author whether it can be dropped. On the one hand, at present we do not know any examples with $\mathrm{Ric}\ge 0$ and $E(M,p)={1}/{2}$. On the other hand, we are unable to show $E(M,p)\not=1/2$ even when $\pi_1(M)=\mathbb{Z}$.

\begin{quest}\label{quest_1/2}
	Let $(M,p)$ be an open $n$-manifold with $\mathrm{Ric}\ge 0$ and a finitely generated fundamental group. Is it true that $E(M,p)<1/2$?
\end{quest}

The converse of Question \ref{quest_1/2} is known to be true: if $E(M,p)<1/2$, then $\pi_1(M)$ is finitely generated (see \cite[Lemma 5]{Sor}). We believe that answering Question \ref{quest_1/2} will lead to a better understanding of the Milnor conjecture \cite{Mil}: the fundamental group of any open $n$-manifold with $\mathrm{Ric}\ge 0$ is finitely generated.
	
We compare Theorem A with previous results. In \cite{Pan_eu,Pan_al_stable}, we have shown that if the universal cover is conic at infinity and satisfies certain stability conditions, (for example, the universal cover has a unique asymptotic cone,) then $\pi_1(M)$ is finitely generated and virtually abelian; in fact, these manifolds have zero escape rate \cite[Corollary 4.7]{Pan_es0}. In contrast, here we do not assume any additional stability conditions in Theorem A(1), and the escape rate may not be equal or close to $0$ in general. In Theorem A(2), if $L$ is sufficiently close to $1$, then the universal cover fulfills the stability condition in \cite{Pan_al_stable} and thus $E(M,p)=0$. Theorem A(2) also confirms \cite[Conjecture 0.2]{Pan_al_stable} on the condition $E(M,p)\not=1/2$.
	
We briefly state our approach to prove Theorem A and give some indications why $\pi_1(M)$ cannot be the discrete Heisenberg $3$-group $H^3(\mathbb{Z})$. Let $\gamma$ be a generator of the center of $H^3(\mathbb{Z})$. $\gamma^b$ can be expressed as a word in terms of two elements $\alpha$ and $\beta$ outside the center; moreover, this word has word length comparable to $b^{1/2}$ as $b\to\infty$. This expression provides an upper bound on the length growth of $\gamma$:
$$|\gamma^b|\le C\cdot b^{1/2}$$
for all $b$ large. If one could find a lower bound violating the above upper bound, then it would end in a desired contradiction. To prove a lower bound, we study the equivariant asymptotic cones of $(\widetilde{M},\langle\gamma\rangle)$:
$$(r_i^{-1}\widetilde{M},\tilde{p},\langle\gamma\rangle)\overset{GH}\longrightarrow (Y,y,H),$$
where $r_i\to\infty$, $\widetilde{M}$ is the universal cover of $M$, and $\langle\gamma\rangle$ is the group generated by $\gamma$. The limit $(Y,y,H)$ may depend on the sequence $r_i$, so we shall study all equivariant asymptotic cones. One key intermediate step is to show that the asymptotic orbit $Hy$ is always homeomorphic to $\mathbb{R}$ (Proposition \ref{R_orb}). This actually requires some understanding of equivariant asymptotic cones of $(\widetilde{M},\pi_1(M,p))$ beforehand. Therefore, we shall first show that for any equivariant asymptotic cone $(Y,y,G)$ of $(\widetilde{M},\pi_1(M,p))$, the asymptotic orbit $Gy$ is homeomorphic to $\mathbb{R}^k$ (Proposition \ref{topol_dim_group}), by applying a critical rescaling argument and the metric cone structure. After knowing the orbit $Hy$ is homeomorphic to $\mathbb{R}$, we can further deduce some uniform controls on $Hy$ by using the metric cone structure (see Lemmas \ref{boom_bound} and \ref{orbit_length_bound}). These estimates lead to an almost linear growth estimate of $\gamma$: for any $s\in(0,1)$, we have  
$$|\gamma^b|\ge C\cdot b^{1-s}$$
for all $b$ large (Theorem \ref{growth_estimate}). The index bound in Theorem A(2) follows from the above almost linear growth estimate and the universal bounds in \cite{KW,Pan_al_stable}. We point out that this almost linear growth estimate and its proof are quite different from the case of small escape rate \cite{Pan_es0,Pan_esgap}, where a stronger almost translation estimate holds and the understanding of $Hy$ is not required (see Remark \ref{rem_compare_small_es} for details).
	
Motivated by \cite{PW_ex} and the work in this paper, we propose the conjecture below.	
	
\begin{conj}\label{conj_nil_dim}
	Let $(M,p)$ be an open $n$-manifold with $\mathrm{Ric}\ge 0$ and $E(M,p)\not=1/2$. Suppose that $\pi_1(M)$ contains a torsion-free nilpotent subgroup of nilpotency length $l$. Then there exist an asymptotic cone $(Y,y)$ of the universal cover and a closed $\mathbb{R}$-subgroup of $\mathrm{Isom}(Y)$ such that the orbit $\mathbb{R}y$ is homeomorphic to $\mathbb{R}$ but has Hausdorff dimension at least $l$.
\end{conj}	

The first examples of asymptotic cones that admit isometric $\mathbb{R}$-orbits with Hausdorff dimension strictly larger than $1$ were discovered in \cite{PW_ex}. It was suspected in \cite[Remark 1.7]{PW_ex} that this feature of extra Hausdorff dimension might be related to the nilpotency of fundamental groups. Conjecture \ref{conj_nil_dim} is a formal description of the question first raised in \cite[Remark 1.7]{PW_ex}.
	
We organize the paper as follows. We start with some preliminaries in Section \ref{sec_pre}; this includes some results of nilpotent isometric actions on metric cones and the escape rate. In Section \ref{sec_nil}, assuming that $M$ has the conditions in Theorem A(1) and $\pi_1(M)$ is nilpotent, we study the equivariant asymptotic geometry of $(\widetilde{M},\pi_1(M,p))$. Then in Section \ref{sec_Z}, we further study the asymptotic orbits coming from $\langle\gamma\rangle$-action, where $\gamma\in \pi_1(M,p)$ has infinite order. The properties of these asymptotic orbits lead to the almost linear growth estimate and virtual abelianness in Section \ref{sec_abel}.
	
\tableofcontents

\textit{Acknowledgements.} The author is supported by Fields Postdoctoral Fellowship from the Fields Institute. The author is grateful to Vitali Kapovitch and Guofang Wei for many helpful conversations during the preparation of this paper.

\section{Preliminaries}\label{sec_pre}

\subsection{Equivariant asymptotic geometry and metric cones}

Let $M$ be an open $n$-manifold with $\mathrm{Ric}\ge 0$. Recall that for any sequence $r_i\to\infty$, after passing to a subsequence if necessary, the corresponding blow-down sequence converges in the pointed Gromov-Hausdorff topology:
$$(r_i^{-1}M,p)\overset{GH}\longrightarrow (X,x).$$
We call the limit space $(X,x)$ an \textit{asymptotic cone} of $M$. In general, the limit $X$ may not be unique and may not be a metric cone (see examples in \cite{CC97}).

Recall that an open $n$-manifold $M$ is said to have \textit{Euclidean volume growth of constant $L$}, if 
$$\lim\limits_{R\to\infty} \dfrac{\mathrm{vol}(B_R(p))}{\mathrm{vol}(B^n_R(0))}=L>0,$$
%$\mathrm{vol}(B_R(p))\ge L\cdot \mathrm{vol}(B^n_R(0))$ for all $R>0$,
where $p\in M$ and $B^n_R(0)$ is an $R$-ball in Euclidean space $\mathbb{R}^n$. By Bishop-Gromov volume comparison, the above limit always exists and is no greater than $1$.

\begin{thm}\cite{CC96}
	Let $M$ be an open $n$-manifold with $\mathrm{Ric}\ge 0$. If $M$ has Euclidean volume growth, then $M$ is conic at infinity.
\end{thm}

For the purpose of understanding fundamental groups, we study the asymptotic geometry of the Riemannian universal cover $\widetilde{M}$ with the isometric $\Gamma$-action, where $\Gamma=\pi_1(M,p)$. Let $r_i\to\infty$ be a sequence. We can pass to a subsequence and consider the pointed equivariant Gromov-Hausdorff convergence \cite{FY}:
$$(r_i^{-1}\widetilde{M},\tilde{p},\Gamma)\overset{GH}\longrightarrow (Y,y,G),$$
where $G$ is a closed subgroup of the isometry group of $Y$. It follows from the work of Colding-Naber that $G$ is always a Lie group \cite{CN}. We call the limit space $(Y,y,G)$ an \textit{equivariant asymptotic cone} of $(\widetilde{M},\Gamma)$.

As a matter of fact, if $M$ has nonnegative sectional curvature, then as a consequence of the soul theorem \cite{CG_soul}, the equivariant asymptotic cone of $(\widetilde{M},\Gamma)$ is unique as $(C(Z),z,G)$, where $C(Z)$ is a metric cone with vertex $z$; moreover, the orbit $Gz$ is an Euclidean factor of $C(Z)$.

Let $\Omega(\widetilde{M},\Gamma)$ be the set of all equivariant asymptotic cones of $(\widetilde{M},\Gamma)$.

\begin{prop}\label{cpt_cnt}
	The set $\Omega(\widetilde{M},\Gamma)$ is compact and connected in the pointed equivariant Gromov-Hausdorff topology.
\end{prop}

See \cite[Proposition 2.1]{Pan_eu} for a proof. The compactness allows us to obtain new spaces in $\Omega(\widetilde{M},\Gamma)$ as the Gromov-Hausdorff limit of any sequence $\{(Y_j,y_j,G_j)\}\subseteq \Omega(\widetilde{M},\Gamma)$. Also, note that if $(Y,y,G)\in\Omega(\widetilde{M},\Gamma)$, then its scaling $(sY,y,G)$ is also an equivariant asymptotic cone for any $s>0$. Therefore, the Gromov-Hausdorff limit of any sequence $(s_jY_j,y_j,G_j)$ in $\Omega(\widetilde{M},\Gamma)$, where $s_j>0$, is an equivariant asymptotic cone as well.

For asymptotic cones that are metric cones, by Cheeger-Colding's splitting theorem \cite{CC96}, any line in the space must split off isometrically. Therefore, we can write such a metric cone $(Y,y)$ as $(\mathbb{R}^k \times C(Z),(0,z))$, where $C(Z)$ is a metric cone without lines and $z$ is the unique vertex of $C(Z)$. The isometry group of $\mathrm{Isom}(Y)$ also splits as 
$$\mathrm{Isom}(Y)=\mathrm{Isom}(\mathbb{R}^k)\times \mathrm{Isom}(C(Z)).$$
After setting a point in $\mathbb{R}^k$ as the origin of $\mathbb{R}^k$, we can express any isometry $g\in \mathrm{Isom}(Y)$ as 
$$g=(A,v,\alpha),$$
where $(A,v)\in \mathrm{Isom}(\mathbb{R}^k)=O(k)\ltimes \mathbb{R}^k$ and $\alpha\in \mathrm{Isom}(C(Z))$. The multiplication in $\mathrm{Isom}(Y)$ is given by
$$(A_1,v_1,\alpha_1)\cdot (A_2,v_2,\alpha_2)=(A_1A_2,A_1v_2+v_1,\alpha_1\alpha_2).$$
Also note that since $z$ is the unique vertex of $C(Z)$, any isometry of $C(Z)$ must fix the vertex $z$. Consequently, for any element $g\in\mathrm{Isom}(Y)$, the orbit point $gy$ must be contained in the Euclidean factor $\mathbb{R}^k\times \{z\}$.

Next, we go through some basic results about nilpotent isometric actions on metric cones. Recall that a group $N$ is called \textit{nilpotent}, if its lower central series terminates at the trivial identity subgroup, that is,
$$N=C_{0}(N)\triangleright C_{1}(N)\triangleright...\triangleright C_{l}(N)=\{e\},$$
where the subgroup $C_{j+1}(N)=[C_{j}(N),N]$ is inductively defined.
The smallest integer $l$ such that $C_{l}(N)=\{e\}$ is called the \textit{nilpotency length} of $N$. 

\begin{lem}\label{isom_eu_commute}
	Let $G$ be a nilpotent subgroup of $\mathrm{Isom}(\mathbb{R}^k)$. Then two elements $(A,v)$ and $(B,w)$ in $G$ commute if and only if $A$ and $B$ commute.
\end{lem}
See \cite[Lemma 2.4]{Pan_al_stable} for a proof.

\begin{lem}\label{cone_isom_central}
	Let $(Y,y)\in\mathcal{M}(n,0)$ be a metric cone with vertex $y$. Let $G$ be a closed nilpotent subgroup of the isometry group of $Y$. Then the center of $G$ has finite index in $G$; in particular, the identity component subgroup $G_0$ must be central in $G$.
\end{lem}

\begin{proof}
	We write $Y=\mathbb{R}^k\times C(Z)$, where $C(Z)$ does not contain lines, and consider the group homomorphism
	\begin{align*}
		\psi: \mathrm{Isom}(Y) &\to O(k)\times \mathrm{Isom}(C(Z))\\
		      (A,v,\alpha) &\mapsto (A,\alpha). 
    \end{align*}	
    Note that $\overline{\psi(G)}$, the closure of $\psi(G)$, is a compact nilpotent Lie group. It follows from a standard result of group theory that the identity component of $\overline{\psi(G)}$, denoted by $K$, is central and of finite index in $\overline{\psi(G)}$ (see, for example, \cite[Lemma 5.7]{Pan_al_stable}) Now we consider the subgroup $H$ of $G$ defined by
    $$H=\psi^{-1}(K) \cap G,$$
    which has finite index in $G$.
    If follows from Lemma \ref{isom_eu_commute} that $H$ is central in $G$.
\end{proof}

Let $(M,p)$ be an open manifold with the assumptions in Theorem A(1) and a nilpotent fundamental group $\Gamma$. Lemma \ref{cone_isom_central} implies that $G$ is always virtually abelian for any $(Y,y,G)\in\Omega(\widetilde{M},\Gamma)$. One may ask whether the virtual abelianness of all asymptotic limit groups $G$ indicates that $\Gamma$ itself should be virtually abelian as well. However, it is possible that $\Gamma$ is a torsion-free nilpotent non-abelian group, while all asymptotic limit groups of $\Gamma$ are abelian; see Appendix for the example. In other words, the nilpotency length of $\Gamma$ may not be well-preserved in the asymptotic limits. 

\subsection{Escape rate}

The notion of escape rate was introduced in \cite{Pan_es0} to study the structure of fundamental groups. It measures where the minimal representing geodesic loops of $\pi_1(M,p)$ are positioned in $M$. We assign two natural quantities to any loop $c$ based at $p\in M$: its length and its size. Here, size means the smallest radius $R$ such that $c$ is contained in the closed ball $\bar{B}_R(p)$, or equivalently, the Hausdorff distance between the loop $c$ and the base point $p$. For each element $\gamma\in \pi_1(M,p)$, we choose a representing geodesic loop of $\gamma$ at $p$, denoted by $c_\gamma$, such that $c_\gamma$ has the minimal length in its homotopy class; if there are multiple choices of $c_\gamma$, we choose the one with the smallest size. We write 
$$|\gamma|:=d(\gamma\tilde{p},\tilde{p})=\mathrm{length}(c_\gamma)$$ for convenience. The escape rate of $(M,p)$ is defined as
$$E(M,p)=\limsup_{|\gamma|\to\infty}\dfrac{\mathrm{size}(c_\gamma)}{\mathrm{length}(c_\gamma)}.$$
As a convention, if $\pi_1(M)$ is a finite group, then we set $E(M,p)=0.$ 

In \cite{Sor}, Sormani proved that if $\pi_1(M)$ is not finitely generated, then there is a consequence of elements $\gamma_i\in \pi_1(M,p)$ with representing geodesic loops $c_i$ that are minimal up to halfway. In other words, if $E(M,p)\not= 1/2$, then $\pi_1(M)$ is finitely generated.

\begin{lem}\label{index_escape_rate}
	Let $(M,p)$ be an open manifold with $\mathrm{Ric}\ge 0$ and let $F: (\hat{M},\hat{p})\to (M,p)$ be a finite cover. Then $E(\hat{M},\hat{p})\le E(M,p)$.
\end{lem}

\begin{proof}
	First note that we can naturally identify $\pi_1(\hat{M},\hat{p})$ as a subgroup of $\pi_1(M,p)$, namely $F_\star (\pi_1(\hat{M},\hat{p}))\subseteq \pi_1(M,p)$ via the injection $F_\star:\pi_1(\hat{M},\hat{p})\to \pi_1(M,p)$. Let $\gamma\in \pi_1(\hat{M},\hat{p})$ and let $\sigma$ be a minimal representing geodesic loop of $\gamma$ at $\hat{p}$. Then $F(\sigma)$ is a minimal representing geodesic loop of $F_\star(\gamma)\in \pi_1(M,p)$. We have
	\begin{align*}
		d_H(p,F(\sigma))&=\inf\{R>0|B_R(p)\supseteq F(\sigma)\}\\
		&=\inf\{R>0|  B_R(F^{-1}(p))\supseteq \sigma \}.
	\end{align*}
	Because $F$ is a finite cover, $F^{-1}(p)$ consists of finitely many points. Let $D>0$ be the diameter of $F^{-1}(p)$, then
	$$B_{R+D}(\hat{p})\supseteq B_R(F^{-1}(p))$$
	for all $R>0$. It follows that
	$$\inf\{R>0|  B_R(F^{-1}(p))\supseteq \sigma \}\ge \inf\{R>0|B_R(\hat{p})\supseteq \sigma\}-D=d_H(\hat{p},\sigma)-D.$$
	
	For a sequence of elements $\gamma_i\in \pi_1(\hat{M},\hat{p})$ and their corresponding minimal representing geodesic loops $\sigma_i$ such that
	$$\lim\limits_{i\to\infty} \dfrac{d_H(\hat{p},\sigma_i)}{\mathrm{length}(\sigma_i)}=E(\hat{M},\hat{p}),$$
	we have
	$$\dfrac{d_H(p,F(\sigma_i))}{\mathrm{length}(F(\sigma_i))}\ge \dfrac{d_H(\hat{p},\sigma_i)}{\mathrm{length}(\sigma_i)}\to E(\hat{M},\hat{p}).$$
	This shows that $E(M,p)\ge E(\hat{M},\hat{p})$.
\end{proof}

With Lemma \ref{index_escape_rate}, we can assume that $\pi_1(M,p)$ is nilpotent without loss of generality when proving Theorem A. In fact, because $E(M,p)\not=1/2$, $\pi_1(M)$ is finitely generated. By \cite{Mil,Gro_poly}, $\pi_1(M)$ has a nilpotent subgroup $N$ of finite index; moreover, according to \cite{KW}, we can assume that the index of $N$ is bounded by some constant $C(n)$. Let $\hat{M}=\widetilde{M}/N$ be an intermediate cover of $M$ and let $\hat{p}\in\hat{M}$ be a lift of ${p}$. Lemma \ref{index_escape_rate} assures that $E(\hat{M},\hat{p})\not=1/2$. In order to prove Theorem A, it suffices to show that $N$ is virtually abelian and further bound the index when $\widetilde{M}$ has Euclidean volume growth.

\section{Asymptotic orbits of nilpotent group actions}\label{sec_nil}

In this section, we always assume that $(M,p)$ is an open $n$-manifold with $\mathrm{Ric}\ge 0$ and $E(M,p)\not=1/2$. Due to Lemma \ref{index_escape_rate}, we will also assume that $\pi_1(M,p)$ is an infinite nilpotent group, denoted by $N$.

The goal of this section is to study the properties of asymptotic equivariant cones of $(\widetilde{M},N)$ when $\widetilde{M}$ is conic at infinity. In particular, we will show that there is an integer $k$ for any $(Y,y,G)\in \Omega(\widetilde{M},N)$, the orbit $Gy$ must be homeomorphic to $\mathbb{R}^k$ (see Proposition \ref{topol_dim_group}).
 
\begin{lem}\label{midpt_orb}
	Let $(Y,y,G)\in\Omega(\widetilde{M},N)$. For any point $gy\in Gy$ that is not $y$, there is a minimal geodesic $\sigma$ from $y$ to $gy$ and an orbit point $g'y\in Gy$ such that
	$$d(m,g'y)<\frac{1}{2}\cdot d(y,gy),$$
	where $m$ is the midpoint of $\sigma$. 
\end{lem}

\begin{proof}
	Let $E=E(M,p)<1/2$. Let $r_i\to\infty$ such that 
	$$(r_i^{-1}\widetilde{M},p,N)\overset{GH}\longrightarrow (Y,y,G)$$
	and let $\gamma_i\in N$ such that $\gamma_i\overset{GH}\to g\in G$ with respect to the above convergence. Let $c_i$ be a sequence of minimal geodesic loops based at $p$ representing $\gamma_i$. By the definition of $E(M,p)$, we have
	$$\limsup_{i\to\infty} \dfrac{d_H(p,c_i)}{\mathrm{length}(c_i)}\le E.$$
	For each $i$, we lift $c_i$ to $\tilde{c_i}$ as a minimal geodesic from $\tilde{p}$ to $\gamma_i\tilde{p}$. Let $R_i=d_H(p,c_i)$ and $d_i=\mathrm{length}(c_i)=\mathrm{length}(\tilde{c_i})$. $R_i$ is also the smallest radius such that $\overline{B_{R_i}(N\tilde{p})}$ covers $\tilde{c_i}$. Passing to a subsequence, we obtain
	$$(r_i^{-1}\widetilde{M},p,N,\tilde{c_i})\overset{GH}\longrightarrow (Y,y,G,\sigma),\quad r_i^{-1}d_i\to d(y,gy),\quad r_i^{-1}R_i\to R.$$
	The above $\sigma$ is a limit minimal geodesic from $y$ to $gy$. Moreover, $\sigma$ is contained in $\overline{B_R(Gy)}$; in particular, let $m$ be the midpoint of $\sigma$, then $d(m,g'y)\le R$ for some $g'\in G$. Thus
	$$d(m,g'y)\le R=\lim\limits_{i\to\infty} r_i^{-1}R_i=\lim\limits_{i\to\infty} \dfrac{r_i^{-1}R_i}{r_i^{-1}d_i}\cdot r_i^{-1}d_i \le E\cdot d(y,gy).$$
\end{proof}

Lemma \ref{midpt_orb} states that for some minimal geodesic $\sigma$ from $y$ to $gy$, its midpoint is closer to $Gy$ than the endpoints of $\sigma$. Next, we show that this property implies the connectedness of $Gy$.

\begin{prop}\label{cnt_orb}
	The orbit $Gy$ is connected for all $(Y,y,G)\in \Omega(\widetilde{M},N)$.
\end{prop}

\begin{proof}
	We argue by contradiction. Suppose that $Gy$ is not connected. Let $\mathcal{C}_0$ be the connected component of $Gy$ containing $y$. Note that $\mathcal{C}_0=G_0y$, where $G_0$ is the identity component subgroup of $G$. Because $G$ is a Lie group and $Gy$ is not connected, there is a different component $\mathcal{C}_1$ of $Gy$ such that
	$$d(\mathcal{C}_0,\mathcal{C}_1)=d(y,\mathcal{C}_1)=\min_{gy\in Gy-\mathcal{C}_0}d(y,gy)>0.$$
	Let $gy\in \mathcal{C}_1$ such that $d(y,gy)=d(y,\mathcal{C}_1)$.
	
	\textbf{Claim:} $d(y,\mathcal{C}_1)=d(\mathcal{C}_0,\mathcal{C}_1)$. In fact, suppose that $z_0\in \mathcal{C}_0$ and $z_1\in\mathcal{C}_1$ such that
	$d(z_0,z_1)<d(y,\mathcal{C}_1)$. We can write $z_0=g_0y$ and $z_1=g_1y$, where $g_0\in G_0$ and $g_1\in G-G_0$. Thus 
	$$d(y,\mathcal{C}_1)>d(g_0y,g_1y)=d(y,g_0^{-1}g_1y).$$
	It follows from the choice of $\mathcal{C}_1$ that 
	$g_0^{-1}g_1y\in \mathcal{C}_0=G_0 y.$
	This leads to $$z_1=g_1y\in G_0y=\mathcal{C}_0;$$
	a contradiction.

	 By Lemma \ref{midpt_orb}, there is a minimal geodesic $\sigma$ from $y$ to $gy$ and a point $g'y\in Gy$ such that
	$$d(m,g'y)< \dfrac{1}{2} \cdot d(y,gy),$$
	where $m$ is the midpoint of $\sigma$. Then
	$$d(y,g'y)\le d(y,m)+d(m,g'y)<d(y,gy)=d(\mathcal{C}_0,\mathcal{C}_1);$$
	$$d(gy,g'y)\le d(gy,m)+d(m,g'y)<d(gy,y)=d(\mathcal{C}_1,\mathcal{C}_0).$$
	By our choice of $\mathcal{C}_1$ as the closest component to $\mathcal{C}_0$, $\mathcal{C}_0$ is also the closest component to $\mathcal{C}_1$. The first inequality above implies $g'y\in\mathcal{C}_0$, while the second one implies $g'y\in \mathcal{C}_1$; a contradiction.
\end{proof}

Starting from Lemma \ref{orbit_commute} below, we will assume that the universal cover $\widetilde{M}$ is conic at infinity for the rest of this section.

\begin{lem}\label{orbit_commute}
	Let $(Y,y,G)\in \Omega(\widetilde{M},N)$. Then $g_1g_2y=g_2g_1y$ for all $g_1,g_2\in G$.
\end{lem}

\begin{proof}
	Let $G_0$ be the identity component subgroup of $G$. Because $Gy$ is connected by Proposition \ref{cnt_orb}, we have $Gy=G_0y$. In other words, for any $g\in G$, we can write $gy=hy$ for some $h\in G_0$. Then any $g\in G$ can be written as the product of an element in $G_0$ and an element in the isotropy subgroup at $y$; namely, $g=h\cdot (h^{-1}g)$, where $h\in G_0$ and $h^{-1}g$ fixes $y$. 
	
	Let $g_1,g_2\in G$. We write 
	$$g_1=h_1\cdot \alpha_1,\quad g_2=h_2\cdot \alpha_2,$$
	where $h_1,h_2\in G_0$ and $\alpha_1,\alpha_2$ belong to the isotropy subgroup at $y$. Because $N$ is nilpotent, $G$ must be nilpotent as well. According to Lemma \ref{cone_isom_central}, $G_0$ is central in $G$. Thus
	$$g_1g_2y=h_1\alpha_1h_2\alpha_2y=h_1h_2y=h_2h_1y=h_2\alpha_2h_1\alpha_1y=g_2g_1y.$$
\end{proof}

\begin{defn}\label{def_type_kd}
	Let $(Y,y,G)$ be a space, where $G$ is a closed nilpotent Lie subgroup of $\mathrm{Isom}(Y)$. Let $T$ be a maximal torus of $G_0$. Let $k\in\mathbb{N}$ and $d\in[0,\infty)$. We say that $(Y,y,G)$ is of \textit{type} $(k,d)$, if the orbit $Gy$ is connected and
	$$\dim G - \dim T =k,\quad \mathrm{diam}(Ty)=d.$$
\end{defn}

%\begin{rem}
%	For a connected nilpotent Lie group $G_0$, its maximal torus is unique...
%	Definition \ref{def_type_kd} also makes for spaces that are not metric cones.
%\end{rem}

\begin{lem}\label{cones_blowdown}
	Let $C(Z)$ be a metric cone with a vertex $z$. Let $G$ be a closed nilpotent subgroup of $\mathrm{Isom}(C(Z))$. Suppose that $(C(Z),z,G)$ is of type $(k,d)$. For a sequence $r_i\to \infty$, we consider corresponding blow-down sequence:
	$$(r_i^{-1}C(Z),z,G)\overset{GH}\longrightarrow (C(Z),z,G').$$
	Then the orbit $G'z$ is a $k$-dimensional Euclidean factor in $C(Z)$.
\end{lem}

\begin{proof}
	Because the orbit $Gz$ is connected and contained in a Euclidean factor of $C(Z)$, it suffices to prove the statement when $C(Z)$ is a Euclidean space $\mathbb{R}^l$ and $G$ is a connected Lie group. By Lemma \ref{cone_isom_central}, $G$ is abelian. We set $z$ as the origin $0$ of $\mathbb{R}^l$, then we can write any element of $G$ in the form of $(A,v)\in SO(l)\ltimes \mathbb{R}^l$. Let 
	$$\psi: \mathrm{Isom}(\mathbb{R}^l)\to SO(l),\quad (A,v)\mapsto A$$
	be natural projection. Because $\psi(G)$ is abelian, we can decompose $\mathbb{R}^k$ into an orthogonal direct sum $E+E^\perp$, where $E$ is the maximal subspace such that $A|_{E}=\mathrm{id}|_{E}$ for all $A\in \psi(G)$. Note that by the above construction, any translation in $E$ and any $g\in G$ must commute.
	
	Let $(A,v)\in G$. We write $v=v_1+v_2$, where $v_1\in E$ and $v_2\in E^\perp$. Let $\delta\in\mathrm{Isom}(\mathbb{R}^l)$ be the translation by $-v_1$. We claim that $\delta g=(A,v_2)$ must have a fixed point. The proof of this claim is by linear algebra. We argue by contradiction. Because $\psi(G)$ is an abelian subgroup of $SO(l)$, we can further decompose $E^\perp$ into an orthogonal direct sum of subspaces with dimension at most $2$:
	$$E^\perp=E^1+E^2+...+E^m$$
	such that each $E^i$ is $\psi(G)$-invariant, where $i=1,...,m$. We write $v_2=\sum_{i=1}^m v^i$, where $v^i\in E^i$.
	By the hypothesis that $(A,v_2)$ does not have fixed points, there exists $j\in\{1,...,m\}$ such that $A|_{E^{j}}=\mathrm{id}|_{E^j}$ and $v^j\not= 0$. From the maximality of $E$ in its definition, we can find some element $(B,w)\in G$ such that $Bv^j\not= v^j$. We write
	$$w=w_1+w_2=w_1+(\textstyle \sum_{i=1}^m w^i),$$
	where $w_1\in E$, $w_2\in E^\perp$, and $w^i\in E^i$. Because $(A,v)$ commutes with $(B,w)$, by direct calculation, we have 
	$$(B-I)v^i=(A-I)w^i$$
	for all $i=1,...,m$. Taking $i=j$, we derive that
	$$0\not = (B-I)v^j=(A-I)w^j=0,$$
	a contradiction. We have verified the claim.
	
%	We define a subset $K$ of $\mathrm{Isom}(\mathbb{R}^l)$ as follows:
%	\begin{align*}
%	K=\{\alpha\in \mathrm{Isom}(\mathbb{R}^l)\ |\ &\alpha=\delta\cdot g \text{ has a fixed point, where $\delta$ is} \\
%	&\text{a translation in $V_1$ and $g\in G$ }\}.
%	\end{align*}
%    We allow $\delta$ to be the trivial translation in the above definition. It is straightforward to verify that $K$ is indeed a subgroup of $\mathrm{Isom}(\mathbb{R}^l)$. Because any element in $K$ has a fixed point, we see that $K$ has a compact closure in $\mathrm{Isom}(\mathbb{R}^l)$. By the definition of $K$, any element $g\in G$ can be written as a product $g=\delta\cdot \alpha=\alpha \cdot \delta$, where $\delta$ is a translation in $V_1$ and $\alpha \in K$; moreover, this expression is unique. These enable us to define a group homomorphism
%    $$F:G\to V_1,\quad g=\delta\cdot\alpha\mapsto \delta.$$
    
    By the claim, any element $g\in G$ can be written as a product $g=\delta \alpha=\alpha \delta$, where $\delta$ is a translation in $E$ and $\alpha$ has a fixed point; moreover, this expression is unique. This enables us to define a group homomorphism
        $$F:G\to E,\quad g=\delta\alpha\mapsto \delta.$$
    Noting that $(\mathbb{R}^l,z,G)$ is of type $(k,d)$, we can write $G=H\times T$, where $H$ is a closed subgroup isomorphic to $\mathbb{R}^k$ and $T$ is a torus subgroup. We remark that the choice of $H$ is not unique in general. It is clear that $F(h)\not= 0$ for all nontrivial element $h\in H$; otherwise, $\langle h\rangle $ would be contained in a compact subgroup of $\mathrm{Isom}(\mathbb{R}^l)$, which is not true. Therefore, $F(H)$ consists of all translations in a $k$-dimensional subspace in $V_1$. After blowing down
    $$(r_i^{-1}\mathbb{R}^l,z,G)\overset{GH}\longrightarrow (\mathbb{R}^l,z,G'),$$ 
    it is clear that the limit orbit $G'z$ is formed exactly by translations in $F(H)$. In particular, $G'z$ is a $k$-dimensional Euclidean subspace of $\mathbb{R}^l$.
\end{proof}

Let $(Y,y,G)\in \Omega(\widetilde{M},N)$. By the proof of \cite[Lemma 3.1]{Pan_es0}, the finite generation of $N$ implies that the orbit $Gy$ is always non-compact. In other words, let $(k,d)$ be the type of $(Y,y,G)$, then we always have $k\ge 1$.

\begin{prop}\label{topol_dim_group}
	Let $(M,p)$ be an open $n$-manifold with the assumptions in Theorem A(1). Suppose that the fundamental group $N$ is an infinite nilpotent group. Then there is an integer $k$ such that all $(Y,y,G)\in \Omega(\widetilde{M},N)$ are of type $(k,0)$.
\end{prop} 

The proof of Proposition \ref{topol_dim_group} is by contradiction and a critical rescaling argument, which implicitly uses the connectedness of $\Omega(\widetilde{M},N)$ (Proposition \ref{cpt_cnt}). This kind of arguments is also used in \cite{Pan_eu,Pan_al_stable,Pan_es0,Pan_esgap}, under different contexts, to prove certain uniform properties among all equivariant asymptotic cones. This method requires an equivariant Gromov-Hausdorff distance gap between certain spaces, which we establish below.

\begin{lem}\label{eGH_gap_flat}
	Given any integer $n\ge 2$, there is a constant $\delta(n)>0$ such that the following holds.
	
	Let $(C(Z_j),z_j)\in\mathcal{M}(n,0)$ be a metric cone with vertex $z_j$ and let $G_j$ be a closed nilpotent subgroup of $\mathrm{Isom}(C(Z_j))$, where $j=1,2$. Suppose that\\
	(1) the orbit $G_1z_1$ is a $k_1$-dimensional Euclidean factor of $C(Z_1)$,\\
	(2) the orbit $G_2z_2$ is connected and is of type $(k_2,d_2)$, where $k_2>k_1$.\\
	Then
	$$d_{GH}((C(Z_1),z_1,G_1),(C(Z_2),z_2,G_2))\ge \delta(n).$$ 
\end{lem}

\begin{proof}
	We set $\delta(n)=\frac{1}{100n^2}$. Suppose that 
	$$d_{GH}((C(Z_1),z_1,G_1),(C(Z_2),z_2,G_2))< \delta(n).$$ 
	
	Let $e_1,...,e_{k_1}\in G_1$ such that their orbit points $\{e_1z_1,...,e_{k_1}z_1\}$ forms an orthogonal basis of $G_1z_1\simeq \mathbb{R}^{k_1}$ and $d(e_jz_1,z_1)=1/n$ for all $j=1,...,k_1$. Let $L$ be the subgroup generated by $\{e_1,...,e_{k_1}\}$, then $Lz_1$ is $1$-dense in $G_1z_1$. Let $e'_1,...,e'_{k_1}\in G_2$ such that each $e'_j$ is $\delta(n)$-close to $e_j$, where $j=1,...,k_1$. Let $L'$ be the subgroup generated by $\{e'_1,...,e'_{k_1}\}$. Though elements in $L'$ may not be commutative, by Lemma \ref{orbit_commute}, the orbit $L'z_2$ can be identified as 
	$$L'z_2=\{\textstyle\prod_{j=1}^{k_1} (e'_j)^{l_j}z_2| l_j\in\mathbb{Z}\}\subseteq G_2z_2.$$
	By the inequality $k_2>k_1$ in the second condition, there exists element $g'\in G_2$ such that 
	$$d(g'z_2,z_2)=d(g'z_2,L'z_2)\in (7,8).$$
	Let $g\in G_1$ that is $\delta(n)$-close to $g'$. Because $Lz_1$ is $1$-dense in $G_1z_1$, there exists some element $h=\prod_{j=1}^{k_1} e_j^{l_j}\in L$ such that $d(hz_1,gz_1)\le 1$. By triangle inequality, we have 
	\begin{align*}
	&d(hz_1,z_1)\le d(hz_1,gz_1)+d(gz_1,z_1)\\
	\le& 1+d(gz_1,g'z_2)+d(g'z_2,z_2)+d(z_2,z_1)\le 10.
	\end{align*}
	According to \cite[Lemma 4.10]{Pan_es0}, $G_1$ acts as translations on $G_1z_1$. Recall that each $e_j\in G_1$ has displacement $1/n$ at $z_1$, thus each $l_j\le 10n$. Together with the choice of $\delta(n)$, we see that $h'=\prod_{j=1}^{k_1} (e'_j)^{l_j}\in G_2$ is $1/10$-close to $h\in G_1$. Thus 
	$$d(g'z_2,L'z_2)\le d(g'z_2,h'z_2)\le d(gz_1,hz_1)+d(gz_1,g'z_2)+d(hz_1,h'z_2)\le 2.$$
	This is a contradiction to $d(g'z_2,L'z_2)>7$ and thus
	 $$d_{GH}((C(Z_1),z_1,G_1),(C(Z_2),z_2,G_2))\ge \delta(n).$$ 
\end{proof}

Now we use Lemma \ref{eGH_gap_flat} and a critical rescaling argument to prove Proposition \ref{topol_dim_group}.

\begin{proof}[Proof of Proposition \ref{topol_dim_group}]
	We argue by contradiction.
	
	\textbf{Claim 1:} Suppose that the statement is not true, then there exist spaces $(Y_1,y_1,G_1)$ and $(Y_2,y_2,G_2)$ in $\Omega(\widetilde{M},N)$ such that its orbit $G_jy_j$ is a Euclidean factor of dimension $k_j$, where $j=1,2$, and $k_1>k_2$.
	
	In fact, if the statement of Proposition \ref{topol_dim_group} fails, then either there exists a space $(W,w,H)\in \Omega(\widetilde{M},N)$ of type $(k,d)$ with $d>0$, or there exists $(W_j,w_j,H_j)\in \Omega(\widetilde{M},N)$, where $j=1,2$, of type $(k_j,0)$ with $k_1>k_2$. For the first case above, we consider the blow-up and blow-down limits of $(W,w,H)$:
	$$(jW,w,H)\overset{GH}\longrightarrow (W,w,H_1),\quad (j^{-1}W,w,H)\overset{GH}\longrightarrow (W,w,H_2)$$
	where $j\to\infty$. Because $d>0$, it is clear that the orbit $H_1w$ is a Euclidean factor with dimension strictly larger than $k$. By Lemma \ref{cones_blowdown}, the orbit $H_2w$ is a $k$-dimensional Euclidean factor. Then $(W,w,H_1)$ and $(W,w,H_2)$ are the desired spaces in Claim 1. For the second case, the blow-up limits of $(W_j,w_j,H_j)$, where $j=1,2$, clearly satisfy the requirements. This proves Claim 1. 
    
    Let $(Y_j,y_j,G_j)\in \Omega(\widetilde{M},N)$, where $j=1,2$, as described in Claim 1. Let $r_i,s_i\to\infty$ such that
    $$(r_i^{-1}\widetilde{M},\tilde{p},N)\overset{GH}\longrightarrow (Y_1,y_1,G_1),\quad (s_i^{-1}\widetilde{M},\tilde{p},N)\overset{GH}\longrightarrow (Y_2,y_2,G_2).$$
    By passing to a suitable subsequence of $r_i$ or $s_i$, we can assume that $t_i:=r_i/s_i\to \infty$. We put
    $$(M_i,q_i,N_i)=(r_i^{-1}\widetilde{M},\tilde{p},N).$$
    Then
    $$(M_i,q_i,N_i)\overset{GH}\longrightarrow (Y_1,y_1,G_1),\quad (t_iM_i,q_i,N_i)\overset{GH}\longrightarrow (Y_2,y_2,G_2).$$
    
    Let $\delta(n)$ be the constant in Lemma \ref{eGH_gap_flat}. For each $i$, we define a set of scales $L_i$ by
    \begin{align*}
    	L_i=\{ l\in [1,t_i]\ |&\ d_{GH}((lM_i,q_i,N_i),(W,w,H))\le \delta(n)/10,   \\
    	&  \text{ where $(W,w,H)\in\Omega(\widetilde{M},N)$ has the orbit $Hw$ }\\ &  \text{ as a Euclidean factor with dimension $<k_1$}     \}.
    \end{align*}
    Recall that in $(Y_2,y_2,G_2)$, the orbit $G_2y$ is a $k_2$-dimensional Euclidean factor with $k_2<k_1$, thus $t_i\in L_i$ for all $i$ large; in particular, $L_i$ is non-empty. We choose $l_i\in L_i$ with $\inf L_i \le l_i \le \inf L_i+1$ as a sequence of critical scales.
    
    \textbf{Claim 2:} $l_i\to \infty$. Suppose that $l_i$ subconverges to a number $l_\infty <+\infty$. Then 
    $$(l_iM_i,q_i,N_i)\overset{GH}\longrightarrow (l_\infty Y_1,y_1,G_1).$$ 
    Recall that the orbit $G_1y_1$ in $(Y_1,y_1,G_1)$ is a $k_1$-dimensional Euclidean factor, thus after scaling by $l_\infty$, the orbit $G_1y_1$ in $(l_\infty Y_1,y_1,G_1)$ is also a $k_1$-dimensional Euclidean factor. On the other hand, since $l_i\in L_i$, each $(l_iM_i,q_i,N_i)$ is $\delta(n)/10$ close to some $(W_i,w_i,H_i)\in \Omega(\widetilde{M},N)$ whose orbit $H_iw_i$ is a Euclidean factor of dimension $<k_1$. It follows that
    $$d_{GH}((l_\infty Y_1,y_1,G_1),(W_i,w_i,H_i))\le \delta(n)/2$$
    for all $i$ large; a contradiction to Lemma \ref{eGH_gap_flat}. This proves Claim 2.
    
    Next, after passing to a convergent subsequence, we consider the rescaling limit
    $$(l_iM_i,q_i,N_i)\overset{GH}\longrightarrow (Y',y',G')\in \Omega(\widetilde{M},N).$$
    Let $(k',d')$ be the type of $(Y',y',G')$. It has the following two possibilities.
    
    \textit{Case 1. $k'\ge k_1$.} Recall that each $(l_iM_i,q_i,N_i)$ satisfies
    $$d_{GH}((l_iM_i,q_i,N_i),(W_i,w_i,H_i))\le\delta(n)/10$$
    for some $(W_i,w_i,H_i)\in \Omega(\widetilde{M},N)$ whose orbit $H_iw_i$ is a Euclidean factor of dimension $<k_1$. Since $(l_iM_i,q_i,N_i)$ converges to $(Y',y',G')$, we have
    $$d_{GH}((Y',y',G'),(W_i,w_i,H_i))\le \delta(n)/2$$
    for all $i$ large, where $G'y'$ is of type $(k',d')$ with $k'\ge k_1$ and $H_iw_i$ is a Euclidean factor of dimension $<k_1$. This contradicts Lemma \ref{eGH_gap_flat}. Thus Case 1 cannot happen.
    
    \textit{Case 2. $k'<k_1$.} We consider the blow-down limit of $(Y',y',G')$:
    $$(j^{-1}Y',y',G')\overset{GH}\longrightarrow (Y',y',H'),$$
    where $j\to\infty$. By Lemma \ref{cones_blowdown}, the orbit $H'y'$ is a Euclidean factor of dimension $k'$. Let $J\in\mathbb{N}$ large such that
    $$d_{GH}((J^{-1}Y',y',G'),(Y',y',H'))\le \delta(n)/100.$$
    Note that
    $$(J^{-1}l_iM_i,q_i,N_i)\overset{GH}\longrightarrow (J^{-1}Y',y',G'),$$
    thus 
    $$d_{GH}((J^{-1}l_iM_i,q_i,N_i),(Y',y',H'))\le\delta(n)/10$$
    for all $i$ large. Because $l_i\to\infty$ and $H'y'$ is is a Euclidean factor of dimension $<k_1$, we conclude that $J^{-1}l_i\in L_i$ for all $i$ large. However, this contradicts our choice of $l_i$ as $\inf L_i \le l_i\le \inf L_i +1$. Thus Case 2 cannot happen.
    
    With all possibilities of $(Y',y',G')$ being ruled out, we reach the desired contradiction and thus complete the proof of statement.
\end{proof}

As a direct consequence of Proposition \ref{topol_dim_group}, in any $(Y,y,G)\in\Omega(\widetilde{M},N)$, any compact subgroup of $G$ must fix the base point $y$. This implies the lemmas below.

\begin{lem}\label{same_orb_point}
	Let $(Y,y,G)\in \Omega(\widetilde{M},N)$ and let $h_1,h_2\in G$. If $h_1^my=h_2^my$ for some integer $m\ge 2$, then $h_1y=h_2y$.
\end{lem}

\begin{proof}
	We first prove that if $h^m y=y$ for some integer $m\ge 2$, then $hy=y$. In fact, let $H\subseteq G$ be the closure of the subgroup generated by $h$. Because $h^m y=y$, the orbit $Hy$ consists of at most $m-1$ many points; in particular, the orbit $Hy$ is closed and bounded. Thus $H$ is a compact subgroup of $G$. Because $(Y,y,G)$ is of type $(k,0)$ by Proposition \ref{topol_dim_group},  $H$ must fix $y$. Thus $hy=y$.
	
	Now, let $h_1,h_2\in G$ such that $h_1^my=h_2^my\not= y$ for some integer $m\ge 2$. By Lemma \ref{orbit_commute},
	$$y=h_1^{-m}h_2^m y=(h_1^{-1}h_2)^m y.$$
	It follows from the previous paragraph that $h_1y=h_2y$.
\end{proof}

\begin{lem}\label{power_outsideR}
	Let $(Y,y,G)\in \Omega(\widetilde{M},N)$. Let $H$ be a closed $\mathbb{R}$-subgroup of $N$ and let $\beta\in G$ such that $\beta y$ is outside of $Hy$. Then $d(\beta^m y, Hy)\to\infty$ as $m\to\infty$.
\end{lem}

\begin{proof}
   We argue by contradiction. Suppose that there is a number $C>0$ such that
   $d(\beta^m y, Hy)\le C$
   for all $m\in\mathbb{Z}$. By the connectedness of $Gy$, we can assume that $\beta\in G_0$ without loss of generality. Because $H$ is central in $G$ by Lemma \ref{cone_isom_central}, we can consider the quotient of $(Y,y,G)$ by $H$-action, denoted by $(Y/H,\bar{y},G/H)$. Let $\bar{\beta}\in G/H$ be the quotient of $\beta$. By hypothesis, we have
   $$d(\bar{\beta}\bar{y},\bar{y})>0,\quad d(\bar{\beta}^m \bar{y},\bar{y})\le C$$
   for all $m\in\mathbb{Z}$. Let $K\subseteq G/H$ be the closure of the subgroup generated $\bar{\beta}$. Then $K$ is a compact subgroup in the identity component of $G/H$ with $$0<\mathrm{diam}(K\bar{y})\le C.$$
   On the other hand, because $G_0$ is abelian and $(Y,y,G)$ is of type $(k,0)$, we can write $G_0=\mathbb{R}^k \times T$, where $T$ is a torus group fixing $y$. After the quotient by the $\mathbb{R}$-subgroup $H$, any compact subgroup in $G_0/H$ must fix $\bar{y}$. A contradiction.
\end{proof}

\section{Asymptotic orbits of $\mathbb{Z}$-actions}\label{sec_Z}

Throughout this section, we always assume that an open manifold $M$ satisfies the assumptions in Theorem A(1) and has an infinite nilpotent fundamental group $N$. We fix an element $\gamma\in N$ with infinite order. We will study the equivariant asymptotic cones of $(\widetilde{M},\langle \gamma \rangle)$. Our first goal of this section is to prove the result below.

\begin{prop}\label{R_orb}
	Any space $(Y,y,H)\in \Omega(\widetilde{M},\langle \gamma \rangle)$ must be of type $(1,0)$; consequently, the orbit $Hy$ is connected and homeomorphic to $\mathbb{R}$.
\end{prop}

We remark that the group $H$ could be strictly larger than $\mathbb{R}$, because $H$ may have a non-trivial isotropy subgroup at $y$. Also, recall that $H$ is a closed subgroup of $\mathrm{Isom}(Y)$, thus the orbit $Hy$ is embedded in $Y$, that is, the subspace topology of $Hy$ matches with the quotient topology from $H/K$, where $K$ is the isotropy subgroup of $H$ at $y$.

Here is the rough idea to prove Proposition \ref{R_orb}: suppose that $(Y,y,H)$ is not of type $(1,0)$, then we shall find a space $(Y',y',G')\in\Omega(\widetilde{M},N)$ violating Proposition \ref{topol_dim_group}. 

We need some preparations first.

\begin{defn}
	Let $G$ be a group. We say a subset $S$ of $G$ is \textit{symmetric}, if $S$ satisfies the following conditions:\\
	(1) $\mathrm{id}\in S$,\\
	(2) if $g\in S$, then $g^{-1}\in S$.
\end{defn}

%\begin{defn}
%	Let $(X_j,x_j,G_j)$ be a space and let $S_j$ be a closed symmetric subset of $G_j$, where $j=1,2$. We assume that 
%	$$d_{GH}((X_1,x_1,G_1),(X_2,x_2,G_2))\le \delta$$
%	for some $\delta>0$.
%	We say that $S_1$ and $S_2$ are $\delta$-close, denoted by
%	$$d_{GH}((X_1,x_1,S_1),(X_2,x_2,S_2))\le \delta,$$
%	if\\
%	(1) for any $g_1\in S_1$, there is $g_2\in S_2$ such that $g_2$ is $\delta$-close to $g_1$,\\
%	(2) for any $g_2\in S_2$, there is $g_1\in S_1$ such that $g_1$ is $\delta$-close to $g_2$.
%\end{defn}
%
%\begin{defn}
%	Let 
%	$$(X_i,x_i,G_i)\overset{GH}\longrightarrow (Y,y,H)$$
%	be a pointed equivariant Gromov-Hausdorff convergent sequence. For each $i$, let $S_i$ be a closed symmetric subset of $G_i$. Let $S$ be a closed symmetric subset of $H$. We say that
%	$$(X_i,x_i,S_i)\overset{GH}\longrightarrow (Y,y,S),$$
%	if $S$ satisfies the followings: \\
%	(1) for any $h\in S$, there is a sequence of isometries $g_i\in G_i$ such that $g_i\overset{GH}\to h$;\\
%	(2) for any convergent sequence of isometries $g_i\in S_i$, its limit isometry $h\in H$ belongs to $S$.
%\end{defn}   

\begin{defn}
	Let $(X_i,x_i,G_i)$ be a pointed equivariant Gromov-Hausdorff convergent sequence with limit $(Y,y,H)$. Recall that this means there is a sequence of triples of $\epsilon_i$-approximation maps $(f_i,\varphi_i,\psi_i)$ (see \cite[Definition 3.3]{FY}). For each $i$, let $S_i$ be a closed symmetric subset of $G_i$. We write $\overline{\varphi_i(S_i)}$ as the closure of $\varphi_i(S_i)$ in $H$. We say that the sequence $S_i$ Gromov-Hausdorff converges to a limit closed symmetric subset $S\subseteq H$, denoted by
	$$(X_i,x_i,S_i)\overset{GH}\longrightarrow (Y,y,S),$$
	if $S$ is the limit of $\overline{\varphi_i(S_i)}$ with respect to the topology on the set of all closed subsets of $H$ induced by the compact-open topology. Equivalently, the closed symmetric subset $S\subseteq H$ satisfies\\
	(1) for any $h\in S$, there is a sequence of isometries $g_i\in S_i$ converging to $h$,\\
	(2) any convergent sequence of isometries $g_i\in S_i$ has the limit $h$ in $S$.
\end{defn}

It follows directly from the proof of \cite[Proposition 3.6]{FY} that we have the precompactness result below.
\begin{prop}
	Let $(X_i,x_i,G_i)$ be a pointed equivariant Gromov-Hausdorff convergent sequence with limit $(Y,y,H)$. For each $i$, let $S_i$ be a closed symmetric subset of $G_i$. Then passing to a subsequence, we have the convergence
	$$(X_i,x_i,S_i)\overset{GH}\longrightarrow (Y,y,S)$$
	for some limit closed symmetric subset $S$ of $H$.
\end{prop}

\begin{defn}\label{def_para_orbit}
	Let $(Y,y,G)\in \Omega(\widetilde{M},N)$ and let $gy\in Gy-\{y\}$. Because the orbit $Gy$ is connected, we can assume $g\in G_0$. 
	Let $\exp$ be the exponential map from the Lie algebra of $G_0$ to the Lie group $G_0$; note that $\exp$ is surjective because $G_0$ is abelian. Then $g=\exp (v)$ for some $v$ in the Lie algebra. We define the following subsets of $Gy$:
	$$P(g)y=\{\exp(tv)y|t\in[-1,1]\},$$
	$$\mathbb{R}(g)y=\{\exp(tv)y| t\in\mathbb{R}\}.$$
\end{defn}  

\begin{lem}\label{para_orbit}
	In Definition \ref{def_para_orbit}, the set $P(g)y$, and thus $\mathbb{R}(g)y$, are uniquely determined by the orbit point $gy$.
\end{lem}

\begin{proof}
	We first show that the set $P(g)y$ is independent of the choice of $v$ in Definition \ref{def_para_orbit}. Suppose that 
	$$g=\exp(v)=\exp(w),$$
	where $v,w$ are elements in the Lie algebra of $G_0$. By Lemma \ref{same_orb_point}, we have 
	$$\exp(\frac{1}{b}v)y=\exp(\frac{1}{b}w)y$$
	for any integer $b\in \mathbb{Z}_+$. Then for any integer $a\in \mathbb{Z}_+$, it follows from Lemma \ref{orbit_commute} that
	$$\exp(\frac{a}{b}v)y=\exp(\frac{a-1}{b}v)\exp(\frac{1}{b}v)y=\exp(\frac{1}{b}w)\exp(\frac{a-1}{b}v)y=...=\exp(\frac{a}{b}w)y.$$
	In other words, we have shown that
	$$\exp(tv)y=\exp(tw)y$$
	holds for all $t\in\mathbb{Q}$. Because $P(g)y$ is the closure of the set $$\{\exp(tv)y|t\in[-1,1]\cap \mathbb{Q} \},$$
	we conclude that $P(g)y$ is independent of the choice of $v$.
	
	Next, we show that $P(g)y$ only depends on the orbit point $gy$, but not the choice of $g\in G_0$. Suppose that $h\in G_0$ such that $gy=hy$. Let $v,w$ be vectors in the Lie algebra of $G_0$ such that
	$$\exp(v)=g,\quad \exp(w)=h.$$   
	Following a similar argument in the first paragraph of the proof and applying Lemmas \ref{orbit_commute} and \ref{same_orb_point}, one can clearly verify the result.
\end{proof}

\begin{lem}\label{sym_closed_multi}
	Let $(Y,y,G)\in \Omega(\widetilde{M},N)$ and let $S$ be a closed symmetric subset of $G$. Suppose that the set $Sy$ satisfies the following properties:\\
	(1) $Sy$ is closed under multiplication, that is, if $g_1, g_2\in S$, then $g_1g_2y\in Sy$;\\
	(2) $Sy$ is bounded.\\
	Then $Sy=\{y\}$. 
\end{lem}  

\begin{proof}
	Let $H$ be the closure of the subgroup generated by $S$. The first assumption implies that $Hy={Sy}$. Because $Hy={Sy}$ is bounded, we conclude that $H$ must be a compact subgroup of $G$. Since $(Y,y,G)\in\Omega(\widetilde{M},N)$ is of type $(k,0)$ by Proposition \ref{topol_dim_group}, $H$ fixes $y$. In other words, we have $Sy=Hy=\{y\}$.
\end{proof}  

We are in a position to prove Proposition \ref{R_orb}.

\begin{proof}[Proof of Proposition \ref{R_orb}]
	Let $r_i\to\infty$ be a sequence. We consider the convergence
	$$(r_i^{-1}\widetilde{M},\tilde{p},N,\langle\gamma\rangle)\overset{GH}\longrightarrow (Y,y,G,H).$$
	We shall show that $(Y,y,H)$ is of type $(1,0)$. 
	
	For each $i$, let 
	$$l_i=\min\{l\in\mathbb{Z}_+|d(\gamma^l\tilde{p},\tilde{p})\ge r_i\}.$$
	and let 
	$$S_\gamma(l_i)=\{\mathrm{id},\gamma^{\pm 1},...,\gamma^{\pm l_i}\}$$ be a sequence of symmetric subsets of $\langle \gamma \rangle$. By triangle inequality, we have 
	$$r_i\le |\gamma^{l_i}|\le r_i+|\gamma|,$$
	Passing to a subsequence, we obtain convergence
	$$(r_i^{-1}\widetilde{M},\tilde{p},\gamma^{l_i},S_\gamma(l_i))\overset{GH}\longrightarrow (Y,y,g,A),$$
	where $A$ is a closed symmetric subset of $H$ and $g\in A$ with $d(gy,y)=1$.
	
	\textbf{Claim 1:} The set $Ay$ contains $P(g)y$. Let $b\in \mathbb{Z}_+$. By the choice of $l_i$,
	$$r_i^{-1}d(\gamma^{\lfloor l_i/b\rfloor}\tilde{p},\tilde{p})\le 1,$$
	where $\lfloor\cdot\rfloor$ means the floor function. Thus the sequence $\gamma^{\lfloor l_i/b\rfloor}$ subconverges to some limit $\alpha\in A$. 
	Because
	$$l_i\le b\cdot \lfloor l_i/b\rfloor < l_i+b,$$
	passing to a subsequence if necessary, we can assume that $b\cdot \lfloor l_i/b\rfloor=l_i+b_0$ for some $b_0=0,...,b$ and all $i$. For this subsequence, we have
	$$(r_i^{-1}\widetilde{M},\tilde{p},\gamma^{\lfloor l_i/b\rfloor},\gamma^{b_0},\gamma^{b\cdot \lfloor l_i/b\rfloor})\overset{GH}\longrightarrow (Y,y,\alpha,g_0,g\cdot g_0),$$
	where $g_0\in A$ fixes $y$; moreover, $g\cdot g_0=\alpha^b$. Thus $\alpha$ satisfies
	$$\alpha^b y=g\cdot g_0y=gy.$$
	It follows from Lemma \ref{same_orb_point} that 
	$$\alpha y=\exp(\frac{1}{b}v) y$$ 
	where $\exp(v)=g$. By construction, the limit symmetric subset $A$ contains the set $\{\mathrm{id},\alpha^{\pm 1},...,\alpha^{\pm b}\}$. Therefore, $Ay$ contains the orbit points
	$$\{y,\exp(\pm \dfrac{1}{b}vy),\exp(\pm \dfrac{2}{b}vy),...,\exp(\pm v)y\}.$$
	Because $b$ is an arbitrary positive integer and $Ay$ is closed, we conclude that $Ay$ contains $P(g)y$.
	
	By Claim 1, the limit orbit $Hy$ must contain $\mathbb{R}(g)y$. To this end, we argue by contradiction to prove Proposition \ref{R_orb}. Suppose that $(Y,y,H)$ is not of type $(1,0)$, then there exists an element $\beta\in H$ such that $\beta y\not\in \mathbb{R}(g)y$. Because $(Y,y,G)$ is of type $(k,0)$, by Lemma \ref{power_outsideR} we have
	$$d(\beta^m y, \mathbb{R}(g)y)\to \infty$$
	as $m\to\infty$. Thus we can choose an element as a power of $\beta$, denoted by $h$, such that $d(hy,\mathbb{R}(g)y)\ge 2$. Let $m_i\to \infty$ such that
	$$(r_i^{-1}\widetilde{M},\tilde{p},\gamma^{m_i})\overset{GH}\longrightarrow (Y,y,h).$$
	Because $d(hy,y)\ge 2$, it is clear that $m_i>l_i$ by our choice of $l_i$.
	
	\textbf{Claim 2:} $m_i/l_i\to \infty.$ Suppose that $m_i/l_i\to C\in[1,\infty)$ for a subsequence. We write
	$$m_i=\lfloor C \rfloor\cdot l_i+ o_i,$$
	where $0\le o_i\le l_i$. Note that
	$$(r_i^{-1}\widetilde{M},\tilde{p},\gamma^{\lfloor C \rfloor\cdot l_i},\gamma^{o_i})\overset{GH}\longrightarrow (Y,y, g^{\lfloor C \rfloor},\delta)$$
	with $g^{\lfloor C \rfloor}y\in\mathbb{R}(g)(y)$ and $\delta y \in Ay$. Since $d(\delta y,y)\le 1$, we have
	$$d(hy,\mathbb{R}(g)y)=d(g^{\lfloor C \rfloor}\delta y,\mathbb{R}(g)y)=d(\delta y,\mathbb{R}(g)y)\le 1.$$
	We result in a contradiction to $d(hy,\mathbb{R}(g)y)\ge 2$. This proves Claim 2.
	
	For each $i$, let 
	$$d_i:=\max \{d(\gamma^k\tilde{p},\tilde{p})\ |\ k=l_i,l_i+1,...,m_i\}\to \infty.$$
	It is clear that $d_i\ge r_i$.
	
	\textbf{Claim 3:} $d_i/r_i\to \infty$. Suppose the contrary, that is, $d_i/r_i\to C \in[1,\infty)$. Let
	$$S_\gamma(m_i)=\{\mathrm{id},\gamma^{\pm 1},...,\gamma^{\pm m_i}\}.$$
	Then we obtain convergence
	$$(d_i^{-1}\widetilde{M},\tilde{p},S_\gamma(l_i),S_{\gamma}(m_i))\overset{GH}\longrightarrow (C^{-1}Y,y,A,B).$$
	Recall that $Ay$ contains $P(g)y$ by Claim 1. Together with Claim 2 that $m_i/l_i\to\infty$, we see that $By$ must contain $\mathbb{R}(g)y$, which is unbounded. On the other hand, by the choice of $d_i$, $By$ should be contained in $\bar{B}_1(y)$; a contradiction. This proves Claim 3.
	
	Next, we consider the convergence
	$$(d_i^{-1}\widetilde{M},\tilde{p},N,\langle\gamma\rangle,\gamma^{m_i},S_\gamma(m_i))\overset{GH}\longrightarrow (Y',y',G',H',h',B').$$
	Due to the choice of $d_i$, it is clear that
	$$d_H(B'y',y')=1.$$
	Also, it follows from Claim 3 that $h'y'=y'$.
	
	\textbf{Claim 4:} The set $B'y'$ is closed under multiplication, that is, if $\beta_1,\beta_2\in B'$, then $\beta_1\beta_2y'\in B'y'$. In fact, let $b_{i,1},b_{i,2}\in [-m_i,m_i]$ be two sequences of integers such that
	$$(d_i^{-1}\widetilde{M},\tilde{p},\gamma^{b_{i,1}},\gamma^{b_{i,2}})\overset{GH}\longrightarrow (Y',y',\beta_1,\beta_2).$$
	If $b_{i,1}+b_{i,2}\in [-m_i,m_i]$, then $\beta_1\beta_2\in B'$ and the claim holds trivially. If not, we can write
	$$b_{i,1}+b_{i,2}=\pm m_i + o_i,$$
	where $o_i\in [-m_i,m_i]$. Let $\beta_0\in B'$ be the limit of $\gamma^{o_i}$ after passing to a convergent subsequence. Then 
	$$\beta_1\beta_2 y'=\lim\limits_{i\to\infty} \gamma^{o_i} \cdot \gamma^{\pm m_i} \tilde{p}=\beta_0(h')^{\pm 1} y'=\beta_0 y'\in B'y'.$$
	
	Lastly, we apply Lemma \ref{sym_closed_multi} to $B'$ and conclude that $B'y'=y'$. We end in a contradiction to $d_H(B'y',y')=1$. This contradiction completes the proof. 
\end{proof}

Let $z\in Hy$ be an orbit point. Because $Hy$ is connected, we can write $z=hy$ for some $h\in H_0$. Let $v$ in the Lie algebra of $H_0$ such that $\exp(v)=h$. For convenience, in the rest of the paper, we will denote the orbit point $\exp(tv)y$ by $(th)y$, where $t\in\mathbb{R}$. By the proof of Lemma \ref{para_orbit}, this point $(th)y$ is independent of the choice of $v$ and $h$. Also, with Proposition \ref{R_orb}, we have $Hy=\mathbb{R}(h)y$.

For the rest of this section, we prove some uniform controls on the path $P(h)y$ that will be used later in Section \ref{sec_abel}.

\begin{lem}\label{boom_bound}
	There exists a constant $C_1=C_1(\widetilde{M},\gamma)$ such that the following holds.
	
	For any $(Y,y,H)\in\Omega(\widetilde{M},\langle\gamma\rangle)$ and any $h\in H_0$ with $d(hy,y)\not= 0$, we have
	$$d((th)y,y)\le C_1\cdot d(hy,y)$$
	for all $t\in[0,1]$.
\end{lem}

\begin{proof}
	Without loss of generality, we assume that $d(hy,y)=1$ by scaling $(Y,y,H)$. We argue by contradiction to prove the lemma. Suppose that we have a sequence of spaces $(Y_j,y_j,H_j)\in \Omega(\widetilde{M},\langle\gamma\rangle)$ and $h_j\in H_j$ with $d(h_jy_j,y_j)=1$, but 
	$$R_j:=\max_{t\in[0,1]}d((th_j)y_j,y_j)\to\infty.$$
	Scaling the sequence by $R_j^{-1}$ and passing to a convergent subsequence, we obtain
	$$(R_j^{-1}Y_j,y_j,H_j)\overset{GH}\longrightarrow (Y',y',H')\in\Omega(\widetilde{M},\langle\gamma\rangle).$$ 
	The hypothesis implies $h_jy_j\overset{GH}\to y'$ with respect to the above convergence. 
	We consider the closed symmetric subset $S_j=\{th_j|t\in[0,1]\}$ of $H_j$ and let $S'\subset H'$ be its limit symmetric subset, that is,
	$$(R_j^{-1}Y_j,y_j,S_j)\overset{GH}\longrightarrow (Y',y',S').$$
	
	We claim that the set $S'y'$ is closed under multiplication; the proof is similar to Claim 4 in the proof of Proposition \ref{R_orb}. In fact, for any $\beta_1,\beta_2\in S'$, we have $t_{j,1},t_{j,2}\in[-1,1]$ such that
	$$(R_j^{-1}Y_j,y_j,t_{j,1}h_j,t_{j,2}h_j)\overset{GH}\longrightarrow (Y',y',\beta_1,\beta_2).$$
	If $t_{j,1}+t_{j,2}\in[-1,1]$, then it is clear that $\beta_1\beta_2\in S'$. If not, we write
	$$t_{j,1}+t_{j,2}=\pm 1 + o_j,$$
	where $o_j\in [-1,1]$. The sequence $o_jh_j\in S_j$ subconverges to a limit $\beta_0\in S'$. Then
	$$\beta_1\beta_2y'=\lim_{j\to\infty} (o_jh_j)\cdot(\pm h_j) y_j=\beta_0y'\in S'y'.$$ 
	 
	Since the set $S'y'$ is closed under multiplication and is contained in $\bar{B}_1(y')$, by Lemma \ref{sym_closed_multi}, we obtain $S'y'=y'$. On the other hand, by the construction of $S_j$ and $R_j$, $S'y'$ must have a point with distance $1$ to $y'$. A contradiction.
\end{proof}

%\begin{rem}
%	It is possible to obtain a constant $C(n,E)$.
%\end{rem}

%\begin{prop}\label{orbit_length_bound}
%	There exists a constant $C_2=C_2(\widetilde{M},\gamma)$ such that the following holds.
%	
%	For any $(Y,y,H)\in\Omega(\widetilde{M},\langle\gamma\rangle)$ and any $h\in H$ with $d(hy,y)\not= 0$, we have
%	$$L\cdot \dfrac{d((\frac{1}{L}h)y,y)}{d(hy,y)}\le C_2$$
%	for all $L\in \mathbb{Z}_+$.
%\end{prop}

\begin{lem}\label{orbit_length_bound}
	Given $s,\epsilon\in(0,1)$, there exists a constant $L_0(\widetilde{M},\gamma,s,\epsilon)$ such that for any $(Y,y,H)\in\Omega(\widetilde{M},\langle \gamma \rangle)$ and any $h\in H_0$ with $d(hy,y)=1$, there exists an integer $2\le L\le L_0$ with
	$$L^{1-s}\cdot d((\frac{1}{L}h)y,y)\le\epsilon.$$
\end{lem}

\begin{proof}
	We argue by contradiction. Suppose that for each integer $L_j=j$, there are $(Y_j,y_j,H_j)\in \Omega(\widetilde{M},\langle \gamma \rangle)$ and $h_j\in H_j$ such that $d(h_jy_j,y_j)=1$ and 
	$$L^{1-s}\cdot d((\frac{1}{L}h_j)y_j,y_j) > \epsilon$$
	for all $2\le L\le L_j$. After passing to a subsequence, we consider the convergence
	$$(Y_j,y_j,H_j,h_j)\overset{GH}\longrightarrow (Y',y',H',h').$$

	\textbf{Claim:} For any integer $L\ge 2$, we have
	$$(Y_j,(\frac{1}{L}h_j)y_j)\overset{GH}\to (Y',(\frac{1}{L}h')y').$$
	In fact, due to Lemma \ref{boom_bound}, there is a constant $C_1$ such that
	$$d((\frac{1}{L}h_j)y_j,y_j)\le C_1$$
	for any integer $L\ge 2.$ Thus after passing to a subsequence, we can assume that $\frac{1}{L}h_j$ converges to some limit isometry $\beta\in H'$ as $j\to\infty$. Note that
	$$\beta^L y'=\lim\limits_{j\to\infty} (\frac{1}{L}h_j)^L y_j = \lim\limits_{j\to\infty} h_j y_j= h'y'.$$
	Applying Lemma \ref{same_orb_point}, we see that $\beta y'= (\frac{1}{L}h')y'$ and the claim follows.
	
    The above claim and the hypothesis together imply that for any integer $L\ge 2$,
	$$L^{1-s}\cdot d((\frac{1}{L}h')y',y') = \lim\limits_{j\to\infty} L^{1-s}\cdot d((\frac{1}{L}h_j)y_j,y_j) \ge \epsilon.$$
	Thus
	$$L\cdot d((\frac{1}{L}h')y',y')\ge L^s \epsilon \to \infty $$
	as $L\to\infty$. This shows that in $(Y',y',H')$, the path $P(h')y'$ from $y'$ to $h'y'$ has infinite length, which cannot be true since $P(h')y'$ comes from an $\mathbb{R}$-orbit of some isometric actions embedded in a Euclidean factor $\mathbb{R}^k$.
\end{proof}

\section{Almost linear growth and virtual abelianness}\label{sec_abel}

We prove the almost linear growth estimate (Theorem \ref{growth_estimate}) and Theorem A in this section.

In Lemma \ref{fraction_convergence} and \ref{length_cut_estimate} below, we always assume that the manifold $M$ satisfies the assumptions in Theorem A(1) and the fundamental group is an infinite nilpotent group $N$. We fix $\gamma$ as an element of infinite order in $N$. The purpose of Lemma \ref{length_cut_estimate} is to transfer Lemma \ref{orbit_length_bound}, as an estimate in the asymptotic limits, to an estimate on $\widetilde{M}$ at large scale.

%We will translate Lemma \ref{orbit_length_bound} as a control of the asymptotic orbits to an estimate of $|\gamma^b|$ for all $b$ sufficiently large (see Lemma \ref{length_cut_estimate}). Then we prove the almost linear growth estimate (Theorem \ref{growth_estimate}) and Theorem A.

\begin{lem}\label{fraction_convergence}
	Let $b_i\to\infty$ be a sequence of positive integers and let $r_i=d(\gamma^{b_i}\tilde{p},\tilde{p})$. 
	We consider the convergence
	$$(r_i^{-1}\widetilde{M},\tilde{p},\langle\gamma\rangle,\gamma^{b_i})\overset{GH}\longrightarrow (Y,y,H,h). $$
	Then for any integer $L\in\mathbb{Z}_+$, we have $$(r_i^{-1}\widetilde{M},\gamma^{\lceil b_i/L \rceil}\tilde{p})\overset{GH}\rightarrow (Y,(\frac{1}{L}h)y),$$
	where $\lceil \cdot \rceil$ means the ceiling function.
\end{lem}

The above statement also holds if one replaces the power $\lceil b_i/L \rceil$ by $\lfloor b_i/L \rfloor$. We use the ceiling function in Lemma \ref{fraction_convergence} for later applications.

\begin{proof}[Proof of Lemma \ref{fraction_convergence}]
	The statement of Lemma \ref{fraction_convergence} is to some extent similar to the Claim in the proof of Lemma \ref{orbit_length_bound}, but at the moment we don't have an estimate similar to Lemma \ref{boom_bound} on the sequence. So, we need to first derive a similar estimate on the distance.
	
	\textbf{Claim}: There is a number $C$ such that
	$$r_i^{-1}d(\gamma^{\lceil b_i/L \rceil}\tilde{p},\tilde{p})\le C$$
	for all $i$. This claim assures that $\gamma^{\lceil b_i/L \rceil}\tilde{p}$ subconverges to some limit point in $Y$. Suppose the contrary, that is, 	
	$$r_i^{-1}d(\gamma^{\lceil b_i/L \rceil}\tilde{p},\tilde{p})\to \infty.$$
	For each $i$, we put 
	$$R_i:=\max_{m=1,...,b_i} d(\gamma^m\tilde{p},\tilde{p}).$$
	If follows from the hypothesis that $r_i^{-1}R_i\to\infty$. We consider an asymptotic cone from the sequence $R_i$:
	$$(R_i^{-1}\widetilde{M},\tilde{p},\langle\gamma\rangle,\gamma^{b_i},S_\gamma(b_i))\overset{GH}\longrightarrow (Y',y',h',B),$$
	where 
	$$S_\gamma(b_i)=\{\mathrm{id},\gamma^{\pm 1},...,\gamma^{\pm b_i}\}$$
	and $d(h'y',y')=0$. By Lemma \ref{sym_closed_multi} and the same argument as Claim 4 in the proof of Proposition \ref{R_orb}, we see that $By'$ is closed under multiplication and thus $By'=y'$. On the other hand, by the construction $S_\gamma(b_i)$ and $R_i$, $By'$ should contain a point with distance $1$ to $y'$. This contradiction verifies the claim.
	
	For convenience, below we write $m_i=\lceil b_i/L \rceil$. With the claim, we can pass to a subsequence such that
	$$(r_i^{-1}\widetilde{M},\tilde{p},\gamma^{m_i})\overset{GH}\longrightarrow (Y,y,\alpha),$$
	where $\alpha\in H$. Since 
	$$Lm_i-L\le b_i \le Lm_i$$
	for each $i$, we can pass to a subsequence such that $b_i=Lm_i-K$, where $K$ is some integer between $0$ and $L$. Thus
	$$\gamma^{b_i}=\gamma^{Lm_i}\cdot \gamma^{-K}\overset{GH}\to \alpha^L\cdot\beta$$
	for some $\beta\in H$ with $\beta y=y$. Recall that $h\in H$ is the limit of $\gamma^{b_i}$. It follows that $\alpha^L\beta=h$ and
	$$(\frac{1}{L}h)^L y=hy=\alpha^L\beta y=\alpha^L y.$$
	Applying Lemma \ref{same_orb_point}, we conclude that $(\frac{1}{L}h) y=\alpha y$, that is, $\gamma^{m_i}\tilde{p}\overset{GH}\to (\frac{1}{L}h) y$.
\end{proof}

\begin{lem}\label{length_cut_estimate}
	Give $s\in(0,1)$, there are constants $L_0=L_0(\widetilde{M},\gamma,s)$ and $R_0=R_0(\widetilde{M},\gamma,s)$ such that for all $b\in\mathbb{Z}_+$ with $|\gamma^b|\ge R_0$, there is some integer $2\le L\le L_0$ with
	$$|\gamma^b|\ge L^{1-s}\cdot|\gamma^{\lceil b/L \rceil}|,$$
	where $\lceil \cdot \rceil$ means the ceiling function.
\end{lem}	

\begin{proof}
	Let $L_0=L_0(\widetilde{M},\gamma,s,1/2)$, the constant in Lemma \ref{orbit_length_bound}. We argue by contradiction to prove the statement. Suppose that there is a sequence $b_i\to\infty$ such that
	$$|\gamma^{b_i}|\le L^{1-s}\cdot |\gamma^{\lceil b_i/L\rceil}|$$
	for all $L=2,...,L_0.$ Let $r_i=|\gamma^{b_i}|\to\infty$. We consider
	$$(r_i^{-1}\widetilde{M},\tilde{p},\langle\gamma\rangle,\gamma^{b_i})\overset{GH}\longrightarrow (Y,y,H,h),$$
	where $h\in H$ satisfies $d(hy,y)=1$. For each integer $L\ge 2$, by Lemma \ref{fraction_convergence}, we have $\gamma^{\lceil b_i/L \rceil}\tilde{p} \overset{GH}\to (\frac{1}{L}h)y$. Together with the hypothesis, we deduce
	$$d((\frac{1}{L}h)y,y)=\lim\limits_{i\to\infty} \dfrac{d(\gamma^{\lceil b_i/L \rceil}\tilde{p},\tilde{p})}{d(\gamma^{b_i}\tilde{p},\tilde{p})}\ge \left(\dfrac{1}{L}\right)^{1-s}$$
	for all $L\in\{2,...,L'\}$. On the other hand, by the choice $L_0=L_0(\widetilde{M},\gamma,s,1/2)$ and Lemma \ref{orbit_length_bound}, we have 
	$$d((\frac{1}{L}h)y,y)\le \dfrac{1}{2}\cdot\left(\dfrac{1}{L}\right)^{1-s}$$
	for some $L\in\{2,...,L_0\}$. A contradiction.
\end{proof}

%\begin{proof}
%	We set $L=\lceil C^{2/\epsilon} \rceil$, where $C=C_2(\widetilde{M},\gamma)$ is the constant in Proposition \ref{orbit_length_bound}. With this $L$, suppose that we cannot find a constant $R$ such that the statement holds, then there is a sequence $b_i\to\infty$ such that
%	$$|\gamma^{b_i}|< L^{1-\epsilon}|\gamma^{\lceil b_i/L \rceil}|.$$
%	Let $r_i=|\gamma^{b_i}|\to\infty$. After passing to a subsequence, we consider the convergence
%	$$(r_i^{-1}\widetilde{M},\tilde{p},\langle \gamma \rangle,\gamma^{b_i})\overset{GH}\longrightarrow (Y,y,H,h)$$
%	with $d(hy,y)=1$. By Lemma \ref{fraction_convergence}, $\gamma^{\lceil b_i/L \rceil}\tilde{p}\in r_i^{-1}\widetilde{M}$ converges to $(\frac{1}{L}h)y\in Y$ associated to the above GH convergent sequence. This convergence gives
%	$$d\left(\left(\frac{1}{L}h\right)y,y\right)=\lim\limits_{i\to\infty} r_i^{-1}|\gamma^{\lceil b_i/L \rceil}|\ge L^{-\left(1-\epsilon\right)}.$$
%	On the other hand, it follows from Proposition \ref{orbit_length_bound} and the choice of $L$ that
%	$$d\left(\left(\frac{1}{L}h\right)y,y\right)\le \dfrac{C}{L}\le L^{(\epsilon/2)-1}=L^{-(1-\epsilon/2)}.$$ 
%	They together result in
%	$$L^{-(1-\epsilon)}\le L^{-(1-\epsilon/2)},$$
%	which clearly cannot hold. This finishes the proof.
%\end{proof}

We are ready to prove the almost linear growth estimate.

\begin{thm}\label{growth_estimate}
	Let $M$ be an open $n$-manifold with the assumptions in Theorem A(1). Suppose that its fundamental group is an infinite nilpotent group, denoted by $N$. Let $\gamma\in N$ be an element of infinite order. Given any $s\in(0,1)$, there are positive constants $C_0=C_0(\widetilde{M},\gamma,s)$ and $P_0=P_0(\widetilde{M},\gamma,s)$ such that 
	$$|\gamma^b|\ge C_0\cdot b^{1-s}$$
	holds for all integers $b>P_0$.
\end{thm}

\begin{proof}
	Let $P_0$ be a large constant such that $|\gamma^b|\ge R_0(\widetilde{M},\gamma,s)$ for all $b\ge P_0$, where $R_0(\widetilde{M},\gamma,s)$ is the corresponding constant in Lemma \ref{length_cut_estimate}.
	
	Let $b>P_0$. By Lemma \ref{length_cut_estimate}, we have
	$$|\gamma^b|\ge L_1^{1-s}\cdot |\gamma^{\lceil b/L_1 \rceil}|$$
	for some integer $2\le L_1 \le L_0$,
	where $L_0=L_0(\widetilde{M},\gamma,s)$ is the constant in Lemma \ref{length_cut_estimate}. If ${\lceil b/L_1 \rceil}< P_0$, then we stop right here. If not, we can apply Lemma \ref{length_cut_estimate} again to find some integer $2\le L_2\le L_0$ such that
	$$|\gamma^b|\ge L_1^{1-s}\cdot |\gamma^{\lceil b/L_1 \rceil}|\ge (L_1L_2)^{1-\epsilon}\cdot |\gamma^{\lceil\lceil b/L_1 \rceil/L_2\rceil}|.$$
	Repeating this process, we eventually derive
	$$|\gamma^b|\ge \left( \textstyle\prod_{j=1}^k L_j \right)^{1-s} \cdot |\gamma^{\lceil...\lceil b/L_1\rceil/L_2.../L_k\rceil}|\ge (\textstyle\prod_{j=1}^k L_j)^{1-s} \cdot r_0,$$
	where $\lceil...\lceil b/L_1\rceil/L_2.../L_k\rceil<P_0$ and $r_0=\min_{m\in \mathbb{Z}_+}|\gamma^m|>0$. Noting that
	$$b/(\textstyle\prod_{j=1}^k L_j)\le \lceil...\lceil b/L_1\rceil/L_2.../L_k\rceil<P_0,$$
	we result in
	$$|\gamma^b|\ge \left(\dfrac{b}{P_0}\right)^{1-s}\cdot r_0=C_0\cdot b^{1-s},$$
	where $C_0=r_0/(P_0^{1-s})$.
\end{proof}

\begin{rem}\label{rem_compare_small_es}
	We compare the almost linear growth estimate and its proof with the methods in the small escape rate case \cite{Pan_esgap}. 
	
	When the escape rate is very small, for any $(Y,y,G)\in\Omega(\widetilde{M},N)$, the orbit $Gy$ is Gromov-Hausdorff close to a Euclidean space (see \cite[Theorem 0.1]{Pan_esgap}). This almost Euclidean orbit implies that an almost translation estimate:
	$$|\gamma^{2b}|\ge 1.9\cdot |\gamma^b|$$
	holds for all $b$ large (see \cite[Lemma 4.7]{Pan_esgap}), which is stronger than the almost linear growth estimate here. Also, \cite{Pan_esgap} does not require a description of $\Omega(\widetilde{M},\langle\gamma\rangle)$; knowing $Gy$ as almost Euclidean orbit is sufficient for its proof.
\end{rem}

To derive virtual abelianness from the almost linear growth in Theorem \ref{growth_estimate}, we require the following standard result from group theory:
\begin{lem}\label{center_index}
	Let $\Gamma$ be a group generated by at most $m$ many elements. Suppose that the commutator subgroup $[\Gamma,\Gamma]$ is finite and has at most $k$ elements. Then the center $Z(\Gamma)$ has index at most $C(k,m)$ in $\Gamma$.
\end{lem}

\begin{proof}
	We include a proof here for readers' convenience. Let $\{\gamma_1,...,\gamma_l\}$ be a set of generators of $\Gamma$, where $l\le m$. Let $Z(\gamma_j)$ be the subgroup consisting of all elements in $\Gamma$ that commute with $\gamma_j$. By assumptions, there are at most $k$ elements in $\Gamma$ conjugating to $\gamma_j$ because
	$$g\gamma_j g^{-1}=[g,\gamma_j]\cdot \gamma_j.$$
	Thus $[\Gamma:Z(\gamma_j)]\le k$. Noting that $$Z(\Gamma)=\cap_{j=1}^l Z(\gamma_j),$$
	we conclude
	$$[\Gamma: Z(\Gamma)]\le k^l\le k^m.$$ 
\end{proof}

\begin{lem}\label{finite_commutator}
	Let $(M,p)$ be an open $n$-manifold with $\mathrm{Ric}\ge 0$ and $E(M,p)\not=\frac{1}{2}$. Suppose that\\
	(1) its Riemannian universal cover is conic at infinity,\\
	(2) $N=\pi_1(M,p)$ is nilpotent.\\
	Then the commutator subgroup $[N,N]$ is finite.
\end{lem}

\begin{proof}
	The proof is similar to \cite[Lemma 4.7]{Pan_al_stable}. The difference is that here we use the almost linear growth estimate in Theorem \ref{growth_estimate} instead of the almost translation estimate in \cite[Lemma 4.5]{Pan_al_stable}. We include the proof for completeness.
	
	Let 
	$$N=C_{0}(N)\triangleright C_{1}(N)\triangleright...\triangleright C_{l}(N)=\{e\}$$
	be the lower central series of $N$.
	We prove the following statement by a reverse induction in $k$: if $C_{k+1}(N)$ is finite, then $C_k(N)$ is also finite. Because $N$ is nilpotent, it suffices to show that any element of the form $[\alpha,\beta]$ has finite order, where $\alpha\in N$ and $\beta\in C_{k-1}(N)$.
	
	We argue by contradiction and suppose that for some $\alpha\in N$ and $\beta\in C_{k-1}(N)$, $[\alpha,\beta]$ has infinite order. By triangle inequality,
	$$|[\alpha^b,\beta^b]|\le 2b(|\alpha|+|\beta|)$$
	for all $b\in\mathbb{Z}_+$. On the other hand, we can apply Theorem \ref{growth_estimate} to obtain a lower bound for large $b$ as follows. We can write 
	$$[\alpha^b,\beta^b]=[\alpha,\beta]^{b^2}\cdot h,$$
	where $h\in C_{k+1}(N)$ (see \cite[Lemma 4.4]{Pan_al_stable}). By the inductive assumption that $C_{k+1}(N)$ is finite, there is $D>0$ such that $|h|\le D$ for all $h\in C_{k+1}(N)$. Let $s=1/4$ and let $P_0=P_0(\widetilde{M},[\alpha,\beta],s)$ be the constant in Theorem \ref{growth_estimate}. Triangle inequality and Theorem \ref{growth_estimate} lead to
	$$|[\alpha,\beta]^{b^2}\cdot h|\ge |[\alpha,\beta]^{b^2}|-|h|\ge C\cdot (b^2)^{1-s}-D$$
	for all $b^2>P_0$, where $C_0$ is independent of $b$. Therefore, we derive that
	$$C_0\cdot b^{2-2s}-D \le 2b(|\alpha|+|\beta|)$$
	holds for all $b$ large. Recall that we have chosen $s=1/4$. Then the above inequality clearly results in a contradiction when $b$ is sufficiently large.
\end{proof}

\begin{proof}[Proof of Theorem A(1)]
	By \cite{Mil,Gro_poly}, we can choose a normal nilpotent subgroup $N$ of $\pi_1(M,p)$ with finite index. Let $\hat{M}=\widetilde{M}/N$ be a covering space of $M$ and let $\hat{p}\in \hat{M}$ be a lift of $p\in M$. By Lemma \ref{index_escape_rate}, $E(\hat{M},\hat{p})\not=1/2$. Applying Lemma \ref{finite_commutator} to $(\hat{M},\hat{p})$, we conclude that $[N,N]$ is finite. Thus the center $Z(N)$ has finite index in $N$ by Lemma \ref{center_index}. Now the result immediately follows since $Z(N)$ has finite index in $\pi_1(M,p)$.
\end{proof}

To prove the universal index bound in Theorem A(2), we use the results below from \cite{KW} and \cite{Pan_al_stable}. 

\begin{thm}\cite{KW}\label{KW_bound}
	Given $n\in\mathbb{N}$, there are constants $C_1(n)$ and $C_2(n)$ such that the following holds.
	
	Let $M$ be an open $n$-manifold of $\mathrm{Ric}\ge 0$ and a finitely generated $\pi_1(M)$. Then\\
	(1) $\pi_1(M)$ can be generated by at most $C_1(n)$ many elements,\\
	(2) $\pi_1(M)$ contains a normal nilpotent subgroup of index at most $C_2(n)$ and nilpotency length at most $n$.
\end{thm}

\begin{thm}\cite{Pan_al_stable}\label{commutator_bound}
	Given $n\in\mathbb{N}$ and $L\in(0,1]$, there exists a constant $C(n,L)$ such that the following holds. 
	
	Let $M$ be an open $n$-manifold of $\mathrm{Ric}\ge 0$. Suppose that\\
	(1) $\widetilde{M}$ has Euclidean volume growth of constant at least $L$,\\
	(2) $\Gamma=\pi_1(M,p)$ is finitely generated and nilpotent with nilpotency length $\le n$,\\
	(3) $\#[\Gamma,\Gamma]$ is finite.\\
	Then $\#[\Gamma,\Gamma]\le C(n,L)$.
\end{thm}

\begin{proof}[Proof of Theorem A(2)]
	According to Theorem \ref{KW_bound}(2), we can choose be a normal nilpotent subgroup $N$ of $\pi_1(M,p)$ of index at most $C_1(n)$ and nilpotency length at most $n$. When $\pi_1(M)$ is finite, surely $[N,N]$ is also finite; when $\pi_1(M)$ is infinite and $E(M,p)\not=1/2$, we apply Lemmas \ref{index_escape_rate} and \ref{finite_commutator} to obtain that $[N,N]$ is finite as well. It follows from Theorem \ref{commutator_bound} that the order of $[N,N]$ is bounded by some constant $C_2(n,L)$. Also, Theorem \ref{KW_bound}(1) gives a bound $C_3(n)$ on the number of generators of $N$. Thus by Lemma \ref{center_index}, we deduce
	$$[N:Z(N)]\le C_4(C_2(n,L),C_3(n))=C_5(n,L).$$
	Therefore, 
	$$[\pi_1(M,p):Z(N)]=[\pi_1(M,p):N]\cdot [N:Z(N)]\le C_1(n)C_5(n,L).$$ 
\end{proof}

%\section{Open questions}

\appendix

\section{A nilpotent group with abelian asymptotic limits}
	
	In this appendix, we slightly modify Wei's example \cite{Wei} to construct an open manifold $M$ with $\mathrm{Ric}>0$ and verify that $M$ satisfies following properties:\\
	(1) $\pi_1(M)$ is the discrete Heisenberg $3$-group; and\\
	(2) for any $(Y,y,G)\in \Omega(\widetilde{M},\pi_1(M,p))$, the limit group $G$ is abelian.\\
	This example demonstrates that the nilpotency length of $\Gamma$ may not be preserved in the asymptotic limits.
		
	Let $\widetilde{N}$ be the simply connected $3$-dimensional Heisenberg group and let $\Gamma$ be the discrete Heisenberg $3$-group, that is,
		$$\widetilde{N}=\left\{ 
		\begin{pmatrix}
			1 & a & c\\
			0 & 1 & b\\
			0 & 0 & 1
		\end{pmatrix}	
		\bigg| a,b,c\in\mathbb{R}  \right\},\quad \Gamma=\left\{ 
		\begin{pmatrix}
			1 & a & c\\
			0 & 1 & b\\
			0 & 0 & 1
		\end{pmatrix}	
		\bigg| a,b,c\in\mathbb{Z}  \right\}\subseteq \widetilde{N}.$$
		The Lie algebra of $\widetilde{N}$ has a basis
		$$X_1=\begin{pmatrix}
			0 & 1 & 0\\
			0 & 0 & 0\\
			0 & 0 & 0
		\end{pmatrix}, \quad X_2=\begin{pmatrix}
			0 & 0 & 0\\
			0 & 0 & 1\\
			0 & 0 & 0
		\end{pmatrix}, \quad X_3=\begin{pmatrix}
			0 & 0 & 1\\
			0 & 0 & 0\\
			0 & 0 & 0
		\end{pmatrix},$$
		with $[X_1,X_2]=X_3$ as the only non-trivial Lie bracket. Given $\alpha>0$ and $\beta\ge 1$, we assign a family of norms $\|\cdot\|_r$, where $r\in[0,\infty)$ is the parameter, on this Lie algebra by
		$$\|X_1\|_r=\|X_2\|_r=(1+r^2)^{-\alpha},\quad \|X_3\|_r=(1+r^2)^{-\frac{\beta}{2}-2\alpha}.$$
		The family of norms $\|\cdot\|_r$ uniquely determines a family of left-invariant Riemannian metrics $\widetilde{g_r}$ on $\widetilde{N}$. $\widetilde{g_r}$ satisfies an almost nonnegative Ricci curvature bound:
		$$\mathrm{Ric}(\widetilde{g_r})\ge -C (1+r^2)^{-\beta},$$
		where $C$ is a positive constant. 
		Let $N_r=(N,g_r)$ be the quotient Riemannian manifold $(\widetilde{N},\widetilde{g_r})/\Gamma$. 
		
		Next, we construct an open Riemannian manifold $(M,g)$ as a warped product
		$$M=[0,\infty) \times_f S^{p} \times N_r,\quad g=dr^2+ f(r)^2ds_{p}^2 + g_r,$$
		where $(S^p,ds_p^2)$ is the standard $p$-dimensional sphere and 
		$$f(r)=r(1+r^2)^{-1/4}.$$
		Following the calculation in \cite{Wei}, one can verify that $(M,g)$ has positive Ricci curvature when $p$ is sufficiently large (depending on $\alpha$ and $\beta$).
		
		Let $p\in M$ at $r=0$. We explain that for $\beta>1$, the above constructed open manifold $(M,p)$ satisfies the required condition (2). Let
		$$\gamma_1=\begin{pmatrix}
			1 & 1 & 0\\
			0 & 1 & 0\\
			0 & 0 & 1
		\end{pmatrix}, \quad \gamma_2=\begin{pmatrix}
			1 & 0 & 0\\
			0 & 1 & 1\\
			0 & 0 & 1
		\end{pmatrix}, \quad \gamma_3=\begin{pmatrix}
			1 & 0 & 1\\
			0 & 1 & 0\\
			0 & 0 & 1
		\end{pmatrix}$$
		be elements in $\pi_1(M,p)=\Gamma$. Following the method in \cite[Lemma 1.1]{PW_ex}, one can verify the length estimates
		$$|\gamma_1^l|=|\gamma_2^l|\sim l^{\frac{1}{1+2\alpha}},$$
		$$|[\gamma_1^l,\gamma_2^l]|=|\gamma_3^{(l^2)}|\sim (l^2)^{\frac{1}{1+\beta+4\alpha}}$$
		holds for all $l$ large.
		When $\beta>1$, $|[\gamma_1^l,\gamma_2^l]|$ is much shorter than $|\gamma_1^l|$ and $|\gamma_2^l|$ as $l\to\infty$. Below, we fix a $\beta>1$. Let $r_i\to\infty$ be a sequence and we consider an equivariant asymptotic cone 
		$$(r_i^{-1}\widetilde{M},\tilde{p},N,\langle\gamma_3\rangle)\overset{GH}\longrightarrow (Y,y,G,H).$$
		By construction, it is clear that $H$ is a closed $\mathbb{R}$-subgroup of $G$. Let $l_i\to\infty$ be a sequence of integers such that
		$$(r_i^{-1}\widetilde{M},\tilde{p},\gamma_1^{l_i},\gamma_2^{l_i})\overset{GH}\longrightarrow(Y,y, g_1,g_2),$$
		where $g_1,g_2\in G$ satisfies $$d(g_1y,y)=d(g_2y,y)=1$$
		It follows from the length estimates that $[g_1,g_2]y=y$. Note that $[g_1,g_2]$ is also the limit of $\gamma_3^{(l_i^2)}$, thus $[g_1,g_2]\in H$. Because $H$ is a closed $\mathbb{R}$-subgroup of $G$, we see that $[g_1,g_2]=\mathrm{id}$; in other words, $G$ is abelian.
		
		As a side note, we mention that by a similar argument in \cite{PW_ex}, one can check that the orbit $Hy$ has Hausdorff dimension $1+\beta+4\alpha\ge 2$. This supports Conjecture \ref{conj_nil_dim}.

\end{document}